\documentclass[numbers,webpdf,imaiai]{ima-authoring-template}%

\graphicspath{{Fig/}}

\theoremstyle{thmstyletwo}%
\newtheorem{theorem}{Theorem}
\newtheorem{remark}{Remark}%

\numberwithin{equation}{section}

\usepackage{bm,empheq,nicefrac}

\newtheorem{lemma}{Lemma}

\allowdisplaybreaks

\begin{document}

\DOI{DOI HERE}
\copyrightyear{2025}
\vol{00}
\pubyear{2025}
\access{Advance Access Publication Date: Day Month Year}
\appnotes{Paper}
\copyrightstatement{Published by Oxford University Press on behalf of the Institute of Mathematics and its Applications. All rights reserved.}
\firstpage{1}


\title[Convergence Analysis of DRLM Method for NS Equations]{Convergence analysis of the dynamically regularized Lagrange multiplier method for the incompressible Navier-Stokes equations}

\author{Cao-Kha Doan and Thi-Thao-Phuong Hoang*
\address{\orgdiv{Department of Mathematics and Statistics}, \orgname{Auburn University}, \orgaddress{\street{Auburn},  \state{AL 36849}, \country{USA}}}}
\author{Lili Ju
\address{\orgdiv{Department of Mathematics}, \orgname{University of South Carolina}, \orgaddress{\street{Columbia}, \state{SC 29208}, \country{USA}}}}
\author{Rihui Lan
\address{\orgdiv{School of Mathematical Sciences}, \orgname{Ocean University of China}, \orgaddress{\state{Qingdao 266100}, \country{China}}}
}


\authormark{C.K. Doan, T.T.P. Hoang, L. Ju, and R. Lan}

\corresp[*]{Corresponding author: \href{email:email-id.com}{tzh0059@auburn.edu}}

\received{Date}{0}{Year}
\revised{Date}{0}{Year}
\accepted{Date}{0}{Year}


\abstract{This paper is concerned with temporal convergence analysis of the recently introduced Dynamically Regularized Lagrange Multiplier (DRLM) method for the incompressible Navier-Stokes equations.  A key feature of the DRLM approach is the incorporation of the kinetic energy evolution through a quadratic dynamic equation involving a time-dependent Lagrange multiplier and a regularization parameter.  We apply the backward Euler method with an explicit treatment of the nonlinear convection term and show the unique solvability of the resulting first-order DRLM scheme.  Optimal error estimates for the velocity and pressure are established based on a uniform bound on the Lagrange multiplier and mathematical induction.  Numerical results confirm the theoretical convergence rates and error bounds that decay with respect to the regularization parameter. 
}


\maketitle

\section{Introduction}
The incompressible Navier-Stokes (NS) equations are among the most fundamental models in fluid dynamics, governing the motion of incompressible viscous fluids. They arise in a wide range of applications, from engineering and meteorology to biology and geophysics. Despite their long history and extensive study, developing accurate, efficient, and energy stable numerical methods for the NS equations has been an active topic in recent years. 

The main difficulty in designing energy stable schemes for the NS equations is the treatment of the nonlinear convection term.  Fully implicit schemes~\cite{Mansour90,Turek96} require solving nonlinear systems, which could be computationally expensive.  Semi-implicit schemes~\cite{Weinan95,Tone04,He-Sun07,Li22,Gar23},  on the other hand, solve only linear systems with either constant or variable coefficients, depending on how the convection term is treated.  Due to their computational efficiency,  in what follows we shall focus on semi-implicit schemes with fully explicit convection.  A popular approach is based on the scalar auxiliary variable (SAV) method, which was originally developed for gradient flow problems~\cite{Shen18,Liu20} and later extended to the NS equations~\cite{L-Dong19,L-Shen20,L-Shen22,Zhang22,Ji24,Lan25}. In~\cite{L-Dong19}, an auxiliary variable associated with the kinetic energy was introduced to reformulate the NS equations into an equivalent form with an extra dynamic equation for the auxiliary variable. The resulting schemes are unconditionally energy diminishing, marking the first time that numerical schemes with explicit treatment of the convection term preserve this key property. Since the auxiliary variable is computed from a nonlinear algebraic equation whose existence and uniqueness of a positive solution are not theoretically guaranteed, later in~\cite{L-Shen22} an exponential decay variable was introduced and only a linear equation needs to be solved for the auxiliary variable at each time step. Based on energy stability results, the authors established temporal error estimates for the first-order exponential SAV scheme in two dimensions. Fully discrete convergence analysis of both first- and second-order exponential SAV schemes was recently carried out in~\cite{Lan25}, where a grad-div stabilization term and the choice of finite element spaces play an essential role. Another class of semi-implicit schemes is the Lagrange multiplier-based approach~\cite{Cheng20,Hou23,Cheng24,Y-Kim22}, where intrinsic physical properties, such as energy dissipation laws, are incorporated into a reformulated system to preserve energy stability in terms of the original variables. We refer to~\cite{Diegel17,Liu23,Liu24,Chen24,Huang24,Yang24,Zhu24,Li25} for various fully and semi-implicit schemes applied to the coupled Cahn-Hilliard-Navier-Stokes system, together with their error analysis.

It is well-known that most SAV methods lead to stability results with respect to certain modified energy,  whereas traditional Lagrange multiplier approaches require the time step size to be sufficiently small.  To overcome these limitations while maintaining computational efficiency,  we proposed in~\cite{Doan25} a novel numerical method, namely the \textit{dynamically regularized Lagrange multiplier} (DRLM) method, for the incompressible NS equations.  A dynamic equation (involving the kinetic energy evolution, a Lagrange multiplier, and a regularization term scaled by a positive  parameter $\theta$) is introduced to reformulate the NS equations into an equivalent system. Based on this reformulation, DRLM schemes using backward differentiation formulas (BDF) with explicit treatment of the convection term can be constructed. In~\cite{Doan25}, the first- and second-order DRLM schemes were shown to be unconditionally energy stable with respect to the original variables. Importantly, we observe numerically that the regularization parameter $\theta$ not only guarantees the unique determination of the Lagrange multiplier but also allows for large time step sizes without affecting accuracy and stability of the numerical solutions.  Although DRLM schemes exhibit robust numerical performance across various benchmark tests~\cite{Doan25}, their rigorous error analysis has remained an open question. It should be noted that the DRLM framework can be naturally extended to the coupled Cahn-Hilliard-Navier-Stokes system~\cite{Doan25b} without requiring additional stabilization terms. 

Convergence analysis of DRLM schemes poses significant challenges compared to most existing approaches. A key difficulty lies in the nonlinear coupling of the Lagrange multiplier with the velocity and pressure through a quadratic equation, leading to a nonlinear error equation that is not easily handled by available error analysis frameworks.
Standard techniques, including those developed for SAV methods~\cite{L-Shen22,Lan25,Li25}, result in only suboptimal convergence rates and thus fail to fully reflect the accuracy of DRLM methods.  In addition, it is crucial to establish error bounds that decay with respect to the regularization parameter $\theta$, as indicated by the numerical results in~\cite{Doan25}. We note that the uniform velocity bounds obtained from energy stability estimates depend on $\theta$ and therefore cannot be used to derive $\theta$-robust error bounds.  Finally, the explicit treatment of the nonlinear convection term introduces additional technical difficulties, especially in three dimensions.

The objective of this work is to carry out temporal error analysis for the first-order DRLM scheme that addresses the aforementioned challenges.  Optimal convergence rates for the velocity and Lagrange multiplier are obtained via a series of intermediate steps that mainly handle the coupling between the two quantities and the error of the Lagrange multiplier.  Gronwall's inequality is applied to both velocity and Lagrange multiplier,  followed by an induction argument to control an extra term arising from explicit treatment of the convection term. Optimal convergence rate for the pressure is then derived based on the inf-sup (or LBB) condition, along with certain bounds on the discrete time derivatives of the velocity and Lagrange multiplier errors.  Note that, while the pressure estimate in the fully discrete setting~\cite{Lan25} follows more directly from the velocity estimate, the semi-discrete pressure error analysis is more involved due to the absence of an inequality associated with the Stokes operator (cf. Remark~\ref{rm:pressure}).  Additionally, in contrast to other approaches that rely heavily on velocity bounds from energy stability, our error estimates are independent of such bounds, and all error constants align well with numerical observations related to the regularization parameter $\theta$ in~\cite{Doan25}.  Moreover,  the techniques developed in this paper are valid in both two and three dimensions. 

The rest of the paper is organized as follows.  In Section~\ref{sec:model}, we present the incompressible NS equations, along with the well-posedness results and regularity assumptions on the data. The DRLM formulation and  discretization schemes are introduced in Section~\ref{sec:DRLM}, followed by important properties of the numerical solution produced by the first-order scheme.  Optimal temporal error estimates for the velocity, Lagrange multiplier, and pressure are derived in Section~\ref{sec:analysis}.  The theoretical results are confirmed by numerical experiments presented in Section~\ref{sec:numeric}.  Finally,  some concluding remarks are given in Section~\ref{sec:conclusion}.

\section{Model problem}\label{sec:model}
For an open, bounded domain $\Omega\subset\mathbb{R}^d$ ($d=2,3$) and a terminal time $T>0$, we consider the following incompressible NS equations:  \vspace{-0.2cm}
\begin{subequations}\label{PDE}
\begin{empheq}[left=\empheqlbrace]{align}
    \bm u_t-\nu\Delta \bm u+(\bm u\cdot\nabla)\bm u+\nabla p&=\bm f,  \quad\text{in } \Omega\times (0,T], \label{PDE_1}\\
    \nabla\cdot\bm u&=0, \,\quad\text{in } \Omega\times (0,T], \label{PDE_2} 
\end{empheq} 
\end{subequations}
subject to the initial condition $\bm u(\cdot,0)=\bm u_0$ and no-slip boundary condition, i.e. $\bm u|_{\partial\Omega}=0$. In~\eqref{PDE}, $\bm u$ and $p$ are the unknown velocity vector and pressure,  respectively,  $\bm f$ is an external force, and $\nu$ represents the kinematic viscosity. 

Let $L^2(\Omega)$, $H^k(\Omega)$, and $H^k_0(\Omega)$ denote the standard Sobolev spaces over $\Omega$, with $\bm L^2(\Omega)$, $\bm H^k(\Omega)$, and $\bm H^k_0(\Omega)$ as their $d$-dimensional extensions, e.g., $\bm H^k(\Omega)=(H^k(\Omega))^d$. The norms on $H^k(\Omega)$ and $\bm H^k(\Omega)$ are indicated by $\|\cdot\|_k$. For $L^2(\Omega)$ and $\bm L^2(\Omega)$, we denote by $(\cdot,\cdot)$ and $\|\cdot\|_0$ the inner product and the associated norm, respectively. Since the pressure is defined up to an additive constant in the NS equations, we define $H^k(\Omega)/\mathbb{R}$ as the quotient space consisting of equivalence classes of elements of $H^k(\Omega)$ differing by constants. When $k=0$, the quotient space is denoted by $L^2(\Omega)/\mathbb{R}$. For $\varphi\in H^k(\Omega)/\mathbb{R}$, its norm is given by $\|\varphi\|_{H^k/\mathbb{R}}=\inf_{c\in\mathbb{R}}\|\varphi+c\|_k. $
Denote by $\bm H$ and $\bm V$ the following Hilbert spaces:  \vspace{-0.2cm}
\begin{align*}
\bm H&=\{\bm u\in\bm L^2(\Omega):\nabla\cdot\bm u=0,\;\bm u\cdot\bm n|_{\partial\Omega}=0\}, \quad \bm V=\{\bm v\in \bm H_0^1(\Omega):\nabla\cdot\bm v=0\}.  \vspace{-0.2cm}
\end{align*}

Let us state a result on the existence and uniqueness of a strong solution to the NS equations~\eqref{PDE}, which can be found, for example, in~\cite{Heywood82,Heywood90,Temam95,Boyer13}.  \vspace{-0.2cm}
\begin{theorem}
Assume $\bm u_0\in \bm V$ and $\bm f\in L^{\infty}(0,T;\bm L^2(\Omega))$. There exists a positive time $T^*$, with $T^*=T$ if $d=2$ and $T^*=T^*(\bm u_0)\le T$ if $d=3$, such that the NS equations~\eqref{PDE} admit a unique strong solution $(\bm u,p)$ satisfying
\begin{align*}
&\bm u\in L^2(0,T^*;\bm H^2(\Omega))\cap{C}([0,T^*];\bm V), \quad \bm u_t\in L^2(0,T^*;\bm L^2(\Omega)),\\
&p\in L^2(0,T^*;H^1(\Omega)/\mathbb{R}).
\end{align*}
Moreover, if $d=2$ or, in the case $d=3$, if $\|\bm u_0\|_1$ and $\|\bm f\|_{L^{\infty}(0,T;\bm L^2(\Omega))}$ are sufficiently small, then the solution $(\bm u,p)$ exists for all $t\in [0,T]$, i.e. $T^*=T$ for $d\in\{2,3\}$, and  
\begin{align} \label{make_assump}
\sup_{t\in [0,T]}\|\bm u(t)\|_1<\infty.  
\end{align}
\end{theorem}
The goal of this paper is to establish temporal error estimates for the DRLM method applied to the NS equations. Toward that end,  throughout the paper we assume $\bm u_0$ and $\bm f$ satisfy the following regularity conditions:
\begin{equation} \label{Assump:1}
\bm u_0\in \bm V\cap \bm H^2(\Omega), \; \text{and} \; \, \bm f\in L^{\infty}(0,T;\bm L^2(\Omega)).
\end{equation}
%

\section{DRLM method and properties}\label{sec:DRLM}
In this section, we recall the DRLM formulation introduced in~\cite{Doan25}, which is equivalent to the original NS equations~\eqref{PDE}. Based on this reformulation, we present the first-order DRLM scheme and analyze its fundamental properties. In particular, we establish the unique solvability of the scheme and derive a uniform bound on the Lagrange multiplier, which plays a crucial role in the subsequent error analysis.  Although the DRLM framework naturally extends to higher-order schemes, their rigorous analysis is highly challenging and remains an open problem. 


\subsection{DRLM formulation and discretization schemes}
The idea of the DRLM approach~\cite{Doan25} is to introduce a time-dependent Lagrange multiplier $q(t)$ with $q(0)=1$ and a constant regularization parameter $\theta>0$ to reformulate the NS equations~\eqref{PDE} equivalently as follows:
\begin{subequations}\label{PDE_new}
\begin{empheq}[left=\empheqlbrace]{alignat=2}
&\bm u_t-\nu\Delta\bm u+q(\bm u\cdot\nabla)\bm u+\nabla p=\bm f,&\quad&\text{ in } \Omega\times (0,T],\\
&\nabla\cdot\bm u=0,&\quad&\text{ in } \Omega\times (0,T],\label{PDE_new_2}\\
&\frac{d\mathcal{K}(\bm u)}{dt}+\theta\frac{dq^2}{dt}=(\bm u_t+q\bm (\bm u\cdot\nabla)\bm u,\bm u),&\quad&\text{ in } (0,T],\label{PDE_new_3}
\end{empheq}
\end{subequations}
under the same initial and boundary conditions. In~\eqref{PDE_new_3}, $\mathcal{K}(\bm u)=\frac 12\|\bm u\|_0^2$ denotes the kinetic energy associated with~\eqref{PDE} and the notation $(\cdot,\cdot)$ represents the standard $\bm L^2$ inner product. 

For temporal discretization, let us consider a uniform partition of the time interval $[0,T]$: $0=t_0<t_1<\cdots<t_N=T$ with the time step size $\tau=\nicefrac TN$.  From the reformulated system~\eqref{PDE_new},  DRLM schemes can be derived using BDF methods with explicit treatment of the nonlinear convection term \cite{Doan25}.  In particular, the first-order DRLM scheme, based on the first-order BDF (i.e., backward Euler) method, is given by 
\begin{subequations}\label{DRLM1}
\begin{empheq}[left=\empheqlbrace]{align}
&\frac{\bm u^{n+1}-\bm u^n}{\tau}-\nu\Delta\bm u^{n+1}+q^{n+1}(\bm u^n\cdot\nabla)\bm u^n+\nabla p^{n+1}=\bm f(t_{n+1}),\label{be1}\\
&\nabla\cdot\bm u^{n+1}=0,\quad \bm u^{n+1}|_{\partial\Omega}=0,\label{be2}\\
&\frac{\mathcal{K}(\bm u^{n+1})-\mathcal{K}(\bm u^n)}{\tau}+\theta\frac{(q^{n+1})^2-(q^n)^2}{\tau}=\left(\frac{\bm u^{n+1}-\bm u^n}{\tau}+q^{n+1}(\bm u^n\cdot\nabla)\bm u^n,\bm u^{n+1}\right).\label{be3} \vspace{0.2cm}
\end{empheq}    
\end{subequations}
We refer to~\cite{Doan25} for the formulation of the second-order DRLM scheme, as well as the unconditional energy stability results for both first- and second-order schemes. Note that the initial conditions for the velocity and Lagrange multiplier are $\bm u^0=\bm u_0$ and $q^0=1$. 
An efficient implementation of the DRLM scheme~\eqref{DRLM1}, which involves solving two generalized Stokes systems and a quadratic equation for the Lagrange multiplier at each time step, is presented in~\cite{Doan25} and thus omitted.


Hereafter, for analysis purposes, we assume that $\theta$ is bounded from below by a positive constant, i.e., $\theta\ge \theta_{\mathrm{tol}}$ for some $\theta_{\mathrm{tol}}>0$. For simplicity, we take $\theta_{\mathrm{tol}}=1$ so that $\theta\ge 1$. It should be noted that when $\theta$ is too small or zero, DRLM schemes produce inaccurate values for the Lagrange multiplier unless the time step size is sufficiently small, which affects the accuracy of the velocity and pressure.  More details about numerical performance of the first- and second-order DRLM schemes with different values of~$\theta$ can be found in \cite{Doan25}.

\subsection{Properties}
We begin with the unique solvability and regularity results of the DRLM scheme~\eqref{DRLM1}. 
\begin{lemma}\label{lemma:!}
The DRLM scheme~\eqref{DRLM1} admits a unique solution $(\bm u^{n},p^n,q^{n})$ with
$$\bm u^n\in \bm V\cap\bm H^2(\Omega),\quad p^n\in H^1(\Omega)/\mathbb{R},\quad q^{n}>0,$$
for every $n\in\{1,2,\ldots,N\}$.
\end{lemma}

\begin{proof}\hspace{0.1cm}
It is clear from the initial conditions and assumption~(\ref{Assump:1}) that $\bm u^0=\bm u_0\in \bm V\cap\bm H^2(\Omega)$ and $q^0=1>0$. Suppose $(\bm u^n,q^n)$ with $\bm u^n\in \bm V\cap\bm H^2(\Omega)$ and $q^n>0$ is uniquely determined from~\eqref{DRLM1} for some nonnegative integer $n$. We consider the following generalized Stokes systems:
\begin{alignat}{2}
&\frac{\bm u_1^{n+1}-\bm u^n}{\tau}-\nu\Delta\bm u_1^{n+1}+\nabla p_1^{n+1}=\bm f(t_{n+1}),&\quad \nabla\cdot\bm u_1^{n+1}=0,\quad \bm u_1^{n+1}|_{\partial\Omega}=0,\label{u1p1}\\
&\frac{\bm u_2^{n+1}}{\tau}-\nu\Delta\bm u_2^{n+1}+(\bm u^n\cdot\nabla)\bm u^n+\nabla p_2^{n+1}=0,&\quad \nabla\cdot\bm u_2^{n+1}=0,\quad \bm u_2^{n+1}|_{\partial\Omega}=0.\label{u2p2}
\end{alignat}
We know that $\frac{1}{\tau}\bm u^n+\bm f(t_{n+1})\in \bm L^2(\Omega)$ and $-(\bm u^n\cdot\nabla)\bm u^n\in \bm L^2(\Omega)$ by the inductive hypothesis and assumption~(\ref{Assump:1}). Based on the well-posedness and regularity results for Stokes systems (see, for instance, \cite{Catta61,Heywood82,Heywood90}), we conclude that~\eqref{u1p1} and~\eqref{u2p2} admit unique solutions, $(\bm u_1^{n+1},p_1^{n+1})$ and $(\bm u_2^{n+1},p_2^{n+1})$, respectively, satisfying 
\begin{align}\label{ui_pi}
\bm u_i^{n+1}\in \bm V\cap\bm H^2(\Omega)\, \text{ and }  \, p_i^{n+1}\in H^1(\Omega)/\mathbb{R} \;\text{ for }\; i=1,2.  
\end{align}
Setting $\bm u^{n+1}=\bm u_1^{n+1}+q^{n+1}\bm u_2^{n+1}$ and plugging it in~\eqref{be3} yields 
\begin{align*}
&\frac{\|\bm u_1^{n+1}+q^{n+1}\bm u_2^{n+1}\|_0^2-\|\bm u^n\|_0^2}{2\tau}+\theta\frac{(q^{n+1})^2-(q^n)^2}{\tau}\\
&=\left(\frac{\bm u_1^{n+1}+q^{n+1}\bm u_2^{n+1}-\bm u^n}{\tau}+q^{n+1}(\bm u^n\cdot\nabla)\bm u^n,\bm u_1^{n+1}+q^{n+1}\bm u_2^{n+1}\right),
\end{align*}
which can be simplified as follows:  
\begin{align}\label{quad}
A_{n+1}(q^{n+1})^2+B_{n+1}q^{n+1}+C_{n+1}=0,
\end{align}
where the quadratic coefficients are given by 
\begin{align*}
A_{n+1}&=\theta+\displaystyle \frac 12\|\bm u_2^{n+1}\|_0^2+\tau\nu\|\nabla \bm u_2^{n+1}\|_0^2,\\
B_{n+1}&=-\left(\bm u_1^{n+1}-\bm u^n,\bm u_2^{n+1}\right)-\tau\left((\bm u^n\cdot\nabla)\bm u^n,\bm u_1^{n+1}\right),\\
C_{n+1}&=-\theta(q^n)^2\displaystyle -\frac 12\|\bm u_1^{n+1}-\bm u^n\|_0^2.
\end{align*}
We observe that $A_{n+1}>0$ and $C_{n+1}<0$ for all $\theta>0$. As a consequence, \eqref{quad} has a unique positive solution, namely $q^{n+1}>0$. Let $p^{n+1}:=p_1^{n+1}+q^{n+1}p_2^{n+1}$. We find that $\bm u^{n+1}\in \bm V\cap\bm H^2(\Omega)$ and $p^{n+1}\in H^1(\Omega)/\mathbb{R}$ due to~\eqref{ui_pi}. Furthermore, $(\bm u^{n+1},p^{n+1},q^{n+1})$ satisfies the DRLM scheme~\eqref{DRLM1}, thereby verifying the existence of a solution to~\eqref{DRLM1} at $t_{n+1}$. 

To prove the uniqueness, we claim that if \eqref{DRLM1} admits a solution $(\bm u^{n+1},p^{n+1},q^{n+1})$, then   \vspace{-0.2cm}
\begin{subequations}\label{decom:key}
\begin{align}
\bm u^{n+1}&=\bm u_1^{n+1}+q^{n+1}\bm u_2^{n+1},\\
p^{n+1}&=p_1^{n+1}+q^{n+1}p_2^{n+1}+c^{n+1},  \vspace{-0.2cm}
\end{align}    
\end{subequations}
for any real number $c^{n+1}$ (i.e.,  the pressure is unique up to  a constant). Indeed, let $\widehat{\bm u}^{n+1}=\bm u^{n+1}-\bm u_1^{n+1}-q^{n+1}\bm u_2^{n+1}$ and $\widehat{p}^{n+1}=p^{n+1}-p_1^{n+1}-q^{n+1}p_2^{n+1}$. The combination of \eqref{be1}, \eqref{be2}, \eqref{u1p1}, and \eqref{u2p2} results in
\begin{subequations}
\begin{empheq}[left=\empheqlbrace]{align}
&\frac{\widehat{\bm u}^{n+1}}{\tau}-\nu\Delta\widehat{\bm u}^{n+1}+\nabla\widehat{p}^{n+1}=\bm 0,\label{bar_til_1}\\
&\nabla\cdot\widehat{\bm u}^{n+1}=0,\quad \widehat{\bm u}^{n+1}|_{\partial\Omega}=0.\label{bar_til_2}
\end{empheq}
\end{subequations}
Taking the $\bm L^2$ inner product of \eqref{bar_til_1} with $\widehat{\bm u}^{n+1}$ and utilizing \eqref{bar_til_2} gives us
\begin{align*}
\frac{\|\widehat{\bm u}^{n+1}\|_0^2}{\tau}+\nu\|\nabla\widehat{\bm u}^{n+1}\|_0^2=0,
\end{align*}
which leads to $\widehat{\bm u}^{n+1}=\bm 0$, and therefore, $\nabla\widehat{p}^{n+1}=\bm 0$ due to~\eqref{bar_til_1}. This means that $\bm u^{n+1}$ and $p^{n+1}$ satisfy the decompositions~\eqref{decom:key}. Hence, if both $(\overline{\bm u}^{n+1},\overline{p}^{n+1},\overline{q}^{n+1})$ and $(\tilde{\bm u}^{n+1},\tilde{p}^{n+1},\tilde{q}^{n+1})$ are the solutions of \eqref{DRLM1}, where $\overline{q}^{n+1}>0$ and $\tilde{q}^{n+1}>0$, then we have  \vspace{-0.2cm}
\begin{subequations}\label{decom:bar}
\begin{align}
\overline{\bm u}^{n+1}&=\bm u_1^{n+1}+\overline{q}^{n+1}\bm u_2^{n+1},\label{decom:bar1}\\
\overline{p}^{n+1}&=p_1^{n+1}+\overline{q}^{n+1}p_2^{n+1}+\overline{c}^{n+1},\end{align}  \vspace{-0.2cm}
\end{subequations}
and 
\begin{subequations}\label{decom:tilde}
\begin{align}
\tilde{\bm u}^{n+1}&=\bm u_1^{n+1}+\tilde{q}^{n+1}\bm u_2^{n+1},\label{decom:tilde1}\\
\tilde{p}^{n+1}&=p_1^{n+1}+\tilde{q}^{n+1}p_2^{n+1}+\tilde{c}^{n+1}.
\end{align}
\end{subequations}
By substituting \eqref{decom:bar1} and \eqref{decom:tilde1} into \eqref{be3} (with $(\bm u^{n+1},q^{n+1})$ replaced by $(\overline{\bm u}^{n+1},\overline{q}^{n+1})$ or $(\tilde{\bm u}^{n+1},\tilde{q}^{n+1})$), we find that $\overline{q}^{n+1}$ and $\tilde{q}^{n+1}$ are positive solutions of \eqref{quad}. Thus, we must have $\overline{q}^{n+1}=\tilde{q}^{n+1}$, and consequently, $\overline{\bm u}^{n+1}=\tilde{\bm u}^{n+1}$ and $\overline{p}^{n+1}=\tilde{p}^{n+1}-\tilde{c}^{n+1}+\overline{c}^{n+1}$.  
\end{proof}

Next, we show that the Lagrange multiplier in the DRLM scheme~\eqref{DRLM1} is uniformly bounded for any time step size $\tau>0$.

\begin{lemma}\label{range:q}
Let $\{q^n\}_{n=0}^N$ be the positive sequence generated by the DRLM scheme~\eqref{DRLM1}. There exists a constant $c_0\ge 1$, depending only on $T$, $\|\bm u_0\|_0$, and $\|\bm f\|_{L^{\infty}(0,T;\bm L^2(\Omega))}$, such that
\begin{align*}
0<q^n\le c_0,\quad n=0,1,\ldots,N.
\end{align*}
\end{lemma}
\begin{proof}\hspace{0.1cm}
Taking the $\bm L^2$ inner product of~\eqref{be1} with $\bm u^{n+1}$ yields
\begin{align}\label{rhs_be3}
\left(\frac{\bm u^{n+1}-\bm u^n}{\tau}+q^{n+1}(\bm u^n\cdot\nabla)\bm u^n,\bm u^{n+1}\right)=-\nu\|\nabla\bm u^{n+1}\|_0^2+\left(\bm f(t_{n+1}),\bm u^{n+1}\right).
\end{align}
Let $c_{\bm f}=\|\bm f\|_{L^{\infty}(0,T;\bm L^2(\Omega))}$. From \eqref{be3} and \eqref{rhs_be3}, we obtain
\begin{align}\label{bound_q}
\begin{aligned}
\frac{\|\bm u^{n+1}\|_0^2-\|\bm u^n\|_0^2}{2\tau}+\theta\frac{(q^{n+1})^2-(q^n)^2}{\tau}&=-\nu\|\nabla\bm u^{n+1}\|_0^2+\left(\bm f(t_{n+1}),\bm u^{n+1}\right)\\
&\le \|\bm f(t_{n+1})\|_0\|\bm u^{n+1}\|_0\le c_{\bm f}\|\bm u^{n+1}\|_0.
\end{aligned}
\end{align}
Summing~\eqref{bound_q} over $i=0,1,\ldots,n$, noting that $q^0=1$, gives us
\begin{align*}
\frac{\|\bm u^{n+1}\|_0^2-\|\bm u^0\|_0^2}{2\tau}+\theta\frac{(q^{n+1})^2-1}{\tau}\le c_{\bm f}\sum_{i=0}^n\|\bm u^{i+1}\|_0,
\end{align*}
or equivalently,
\begin{align}\label{bound_q_2}
\|\bm u^{n+1}\|_0^2+2\theta(q^{n+1})^2\le \|\bm u^0\|_0^2+2\theta+2c_{\bm f}\tau\sum_{i=0}^n\|\bm u^{i+1}\|_0.
\end{align}
Let $\eta_n:=\max_{1\le i\le n}\|\bm u^i\|_0$ for $n\in\{1,2,\ldots,N\}$. It follows from~\eqref{bound_q_2} that
\begin{align}\label{bound_q_3}
\|\bm u^{n+1}\|_0^2+2\theta(q^{n+1})^2\le \|\bm u^0\|_0^2+2\theta+2c_{\bm f}T \eta_{n+1}.
\end{align}
Since $\{\eta_n\}$ is a non-decreasing sequence, it is implied from~\eqref{bound_q_3} for $0\le i\le n$ that
\begin{align*}
\|\bm u^{i+1}\|_0^2\le \|\bm u^0\|_0^2+2\theta+2c_{\bm f}T \eta_{i+1}\le \|\bm u^0\|_0^2+2\theta+2c_{\bm f}T \eta_{n+1},
\end{align*}
which leads to 
\begin{align}\label{bound_q_4}
\eta_{n+1}^2\le \|\bm u^0\|_0^2+2\theta+2c_{\bm f}T \eta_{n+1}.
\end{align}
Note that the quadratic equation $\eta_{n+1}^2=\|\bm u^0\|_0^2+2\theta+2c_{\bm f}T \eta_{n+1}$ has two distinct solutions of opposite signs. Therefore, $\eta_{n+1}$ satisfying \eqref{bound_q_4} is bounded by the positive solution, i.e.,
\begin{align}\label{bound_q_5}
\eta_{n+1}\le c_{\bm f}T+\sqrt{(c_{\bm f}T)^2+\|\bm u^0\|_0^2+2\theta}\le 2c_{\bm f}T+\|\bm u^0\|_0+\sqrt{2\theta}.
\end{align}
The combination of~\eqref{bound_q_3} and \eqref{bound_q_5} results in
\begin{align}
\|\bm u^{n+1}\|_0^2+2\theta(q^{n+1})^2\le \|\bm u^0\|_0^2+2\theta+2c_{\bm f}T(2c_{\bm f}T+\|\bm u^0\|_0+\sqrt{2\theta}).
\end{align}
In particular, noting that $\theta\ge 1$, we have
\begin{align*}
(q^{n+1})^2&\le 1+\frac{\|\bm u^0\|_0^2+2c_{\bm f}T(2c_{\bm f}T+\|\bm u^0\|_0+\sqrt{2\theta})}{2\theta} \le 1+\frac{\|\bm u^0\|_0^2+2c_{\bm f}T(2c_{\bm f}T+\|\bm u^0\|_0+\sqrt{2})}{2}.
\end{align*}
The proof  is thus completed.
\end{proof}

\section{Error analysis}\label{sec:analysis}
Let $\{\bm u^n,p^n,q^n\}$ be the numerical solution at time $t_n$ obtained by the DRLM scheme~\eqref{DRLM1}, we define the errors of the velocity, pressure, and Lagrange multiplier as follows: 
\begin{align*}
\begin{aligned}
e_{\bm u}^{n}&=\bm u^{n}-\bm u(t_{n}), \; \, e_p^{n}=p^n-p(t_n), \; \text{and} \; \, e_q^n=q^n-1.
\end{aligned}
\end{align*}
In particular, we have $e_{\bm u}^0=\bm 0$ and $e_q^0=0$. Before stating the main results of this work, let us first recall some useful inequalities.

\subsection{Preliminary inequalities}
Since the homogeneous Dirichlet boundary condition is imposed for the velocity field,  we invoke the Poincar\'{e} inequality and obtain, for some constant $c_1$ depending only on the domain $\Omega$:
\begin{align}
\|\bm v\|_1&\le (1+c_1)\|\nabla\bm v\|_0,\quad \forall\,\bm v\in\bm H_0^1(\Omega).\label{poincare:1}
\end{align}

For $\bm u,\bm v,\bm w\in\bm H_0^1(\Omega)$, we define the trilinear form $b(\cdot,\cdot,\cdot)$ by
\begin{align*}
b(\bm u,\bm v,\bm w)=\int_{\Omega}\left((\bm u\cdot\nabla)\bm v\right)\cdot\bm w\,d\bm x.
\end{align*}
It can be shown that $b$ is skew-symmetric with respect to its last two arguments, i.e., 
\begin{align}\label{b:skew}
b(\bm u,\bm v,\bm w)=-b(\bm u,\bm w,\bm v),\quad \forall\,\bm u\in \bm V,\;\forall\, \bm v,\bm w\in \bm H_0^1(\Omega).
\end{align}
By setting $\bm w=\bm v$ in~\eqref{b:skew}, we have $b(\bm u,\bm v,\bm v)=0,\;\forall\, \bm u\in \bm V,\;\forall\, \bm v\in \bm H_0^1(\Omega)$. Furthermore, for $d\le 3$, there exists a constant $c_2$ depending on $\Omega$ such that (cf.~\cite{Shen92,Temam95,Tone04,Boyer13,L-Shen22}) \vspace{-0.2cm}
\begin{align}\label{b:original}
    |b(\bm u,\bm v,\bm w)|&\le c_2\begin{cases}
        \|\bm u\|_1\|\bm v\|_1\|\bm w\|_1,\\ 
        \|\bm u\|_0\|\bm v\|_2\|\bm w\|_1,\\ 
        \|\bm u\|_1\|\bm v\|_2\|\bm w\|_0,\\ 
        \|\bm u\|_0\|\bm v\|_1\|\bm w\|_2,\\ 
        \|\bm u\|_2\|\bm v\|_1\|\bm w\|_0.\\ 
    \end{cases}
    \end{align}

A discrete version of Gronwall's inequality (see, for example, \cite{Diegel17,He-Sun07,Qua08}) is stated below:
\begin{lemma}\label{lemma:Gronwall}
Let $\{\tilde a_n\}_{n=0}^K$, $\{\tilde b_n\}_{n=1}^K$, $\{\tilde c_n\}_{n=1}^K$, and $\{\omega_n\}_{n=0}^K$ be nonnegative sequences with $\omega_0 \tilde a_0+\tilde c_1\le \tilde c_1e^{\omega_0}$. If $\{\tilde c_n\}$ is non-decreasing and   \vspace{-0.2cm}
\begin{align*}
\tilde a_n+\tilde b_n\le \tilde c_n+\sum_{i=0}^{n-1}\omega_i \tilde a_i,\quad \forall\, 1\le n\le K,  \vspace{-0.2cm}
\end{align*}
then we have  \vspace{-0.2cm}
\begin{align}\label{gronwall:KL}
\tilde a_n+\tilde b_n\le \tilde c_n\mathrm{exp}\left(\sum_{i=0}^{n-1}\omega_i\right),\quad \forall\,1\le n\le K.  \vspace{-0.2cm}
\end{align}
\end{lemma}
We note that the condition $\omega_0 \tilde a_0+\tilde c_1\le \tilde c_1e^{\omega_0}$ in Lemma~\ref{lemma:Gronwall} is needed to ensure~\eqref{gronwall:KL} holds for $n=1$. This condition is satisfied if $\tilde a_0\le \tilde c_1$.

\subsection{Optimal error estimates for the velocity and Lagrange multiplier}
In what follows, we assume that the solution $(\bm u,p)$ to the NS equations~\eqref{PDE} satisfies the following regularity conditions:  \vspace{-0.2cm}
\begin{align}\label{regularity:must}
\begin{aligned}
&\bm u\in L^{\infty}(0,T;\bm H^2(\Omega)),\; \bm u_t\in L^2(0,T;\bm H^1(\Omega)),\\
&\bm u_{tt}\in L^2(0,T;\bm H^{-1}(\Omega)),\; p\in L^2(0,T;H^1(\Omega)/\mathbb{R}).  \vspace{-0.2cm}
\end{aligned}
\end{align}
Let $M$ be a positive constant such that  \vspace{-0.2cm}
\begin{align}\label{M:velo}
\max\left\{\sup_{t\in [0,T]}\|\bm u\|_2,\;\int_0^T\|\bm u_t\|_1^2\,dt,\;\int_0^T\|\bm u_{tt}\|_{-1}^2\,dt\right\}\le M.   \vspace{-0.2cm}
\end{align}

\begin{theorem}\label{thm:DRLM1}
Let $\{\bm u^n\}$, $\{p^n\}$, and $\{q^n\}$ be generated by the first-order DRLM scheme~\eqref{DRLM1}. There exist positive constants $\tau_0$, $C_1$, and $C_2$ depending on $\Omega$, $T$, $\nu$, $M$, $\bm u_0$, and $\bm f$ but independent of $\tau$, $\theta$, and $n$ such that the following error estimates hold for all $\tau\le \tau_0$:   
\begin{align*}
&\|e_{\bm u}^{n+1}\|_0^2+\sum_{i=0}^n\|e_{\bm u}^{i+1}-e_{\bm u}^i\|_0^2+\tau\nu\sum_{i=0}^n\|\nabla e_{\bm u}^{i+1}\|_0^2\le C_1\left(1+\frac{1}{\theta}\right)\tau^2,\quad 0\le n\le N-1,  \vspace{-0.2cm}
\end{align*}
and  \vspace{-0.2cm}
\begin{align*}
|e_q^n|\le \frac{C_2}{\theta}\tau,\quad 0\le n\le N.  \vspace{-0.2cm}
\end{align*}
\end{theorem}
\begin{proof}\hspace{0.1cm} 
To derive  optimal error estimates for both the velocity and the Lagrange multiplier, we perform the following steps. A velocity estimate is first obtained based on its error equation (integrated against $e_{\bm u}^{n+1}$) and the use of the viscosity term to control the truncation error and nonlinear convection terms. Since the resulting estimate involves the Lagrange multiplier error, we bound it in terms of the velocity error. This is a challenging task due to the presence of both $q^{n+1}$ and $(q^{n+1})^2$ in the DRLM scheme and requires combining the dynamic equation for $q$ with the momentum equation and the velocity error equation, tested with $\bm u^{n+1}-\bm u^n$ and $e_{\bm u}^{n+1}-e_{\bm u}^n$, respectively.  Once these estimates are established,  the discrete Gronwall's inequality is applied to both the velocity and the Lagrange multiplier, followed by an inductive argument to achieve optimal convergence rates of the concerned quantities under an $\mathcal{O}(1)$ time step size requirement.


\textbf{Step 1: Estimates for the velocity.} 
Subtracting the momentum equation~\eqref{PDE_1} at $t_{n+1}$ from~\eqref{be1}, we derive the error equation for the velocity and pressure:
\begin{align}\label{err_eq_u}
\frac{e_{\bm u}^{n+1}-e_{\bm u}^n}{\tau}-\nu\Delta e_{\bm u}^{n+1}+\nabla e_p^{n+1}=\bm R_{\bm u}^{n+1}+(\bm u(t_{n+1})\cdot\nabla)\bm u(t_{n+1})-q^{n+1}(\bm u^n\cdot\nabla)\bm u^n,
\end{align}
where the truncation error $\bm R_{\bm u}^{n+1}$ is given by  \vspace{-0.2cm}
\begin{align*}
\bm R_{\bm u}^{n+1}=\bm u_t(t_{n+1})-\frac{\bm u(t_{n+1})-\bm u(t_n)}{\tau}=\frac{1}{\tau}\int_{t_n}^{t_{n+1}}(t-t_n)\bm u_{tt}(t)\,dt.
\end{align*}
Integrating \eqref{err_eq_u} against $e_{\bm u}^{n+1}$ and noting that $(\nabla e_p^{n+1},e_{\bm u}^{n+1})=0$ due to $\nabla\cdot e_{\bm u}^{n+1}=0$ (cf. the incompressibility conditions~\eqref{PDE_new_2} and~\eqref{be2}), we find
\begin{align}\label{est:velo}
\begin{aligned}
&\frac{\|e_{\bm u}^{n+1}\|_0^2-\|e_{\bm u}^n\|_0^2}{2\tau}+\frac{\|e_{\bm u}^{n+1}-e_{\bm u}^n\|_0^2}{2\tau}+\nu\|\nabla e_{\bm u}^{n+1}\|_0^2\\
&\quad =(\bm R_{\bm u}^{n+1},e_{\bm u}^{n+1})+b(\bm u(t_{n+1}),\bm u(t_{n+1}),e_{\bm u}^{n+1})-q^{n+1}b(\bm u^n,\bm u^n,e_{\bm u}^{n+1}).
\end{aligned}
\end{align}
Since $\bm u(t_{n+1})\in \bm H_0^1(\Omega)$ and $\bm u^{n+1}\in \bm H_0^1(\Omega)$ by Lemma~\ref{lemma:!}, we have $e_{\bm u}^{n+1}\in \bm H_0^1(\Omega)$. Hence, the first term on the right-hand side of~\eqref{est:velo} can be bounded by
\begin{align}\label{est:Ru_}
(\bm R_{\bm u}^{n+1},e_{\bm u}^{n+1})\le \left\|\frac{1}{\tau}\int_{t_n}^{t_{n+1}}(t-t_n)\bm u_{tt}\,dt\right\|_{-1}\|e_{\bm u}^{n+1}\|_1\le \left(\int_{t_n}^{t_{n+1}}\|\bm u_{tt}\|_{-1}dt\right)\|e_{\bm u}^{n+1}\|_1.
\end{align}
The last two terms on the right-hand side of~\eqref{est:velo} can be split as follows
\begin{align}\label{B1234}
b(\bm u(t_{n+1}),\bm u(t_{n+1}),e_{\bm u}^{n+1})-q^{n+1}b(\bm u^n,\bm u^n,e_{\bm u}^{n+1})=B_1+B_2+B_3+B_4,
\end{align}
where
\begin{align*}
B_1&=b(\bm u(t_{n+1})-\bm u(t_n),\bm u(t_{n+1}),e_{\bm u}^{n+1})+b(\bm u(t_n),\bm u(t_{n+1})-\bm u(t_n),e_{\bm u}^{n+1}),\\
B_2&=-q^{n+1}b(e_{\bm u}^n,\bm u(t_n),e_{\bm u}^{n+1})-q^{n+1}b(\bm u(t_n),e_{\bm u}^n,e_{\bm u}^{n+1}),\\
B_3&=-q^{n+1}b(e_{\bm u}^n,e_{\bm u}^n,e_{\bm u}^{n+1}),\\
B_4&=-e_q^{n+1}b(\bm u(t_n),\bm u(t_n),e_{\bm u}^{n+1}).
\end{align*}
It follows from properties of the operator  $b$ in~\eqref{b:skew}-\eqref{b:original} and the regularity assumption~\eqref{M:velo} that
\begin{align}\label{est:B1_}
\begin{aligned}
B_1&=b(\bm u(t_{n+1})-\bm u(t_n),\bm u(t_{n+1}),e_{\bm u}^{n+1})-b(\bm u(t_n),e_{\bm u}^{n+1},\bm u(t_{n+1})-\bm u(t_n)) \\
&\le c_2\|\bm u(t_{n+1})-\bm u(t_n)\|_0\|\bm u(t_{n+1})\|_2\|e_{\bm u}^{n+1}\|_1+c_2\|\bm u(t_n)\|_2\|e_{\bm u}^{n+1}\|_1\|\bm u(t_{n+1})-\bm u(t_n)\|_0 \\
&\le 2c_2M\|\bm u(t_{n+1})-\bm u(t_n)\|_0\|e_{\bm u}^{n+1}\|_1\\
&\le 2c_2M\left(\int_{t_n}^{t_{n+1}}\|\bm u_t\|_0\,dt\right)\|e_{\bm u}^{n+1}\|_1.
\end{aligned}
\end{align}
Similarly, noting that $0<q^{n+1}\le c_0$ from Lemma~\ref{range:q}, we obtain
\begin{align}
&\begin{aligned}
B_2&=-q^{n+1}b(e_{\bm u}^n,\bm u(t_n),e_{\bm u}^{n+1})+q^{n+1}b(\bm u(t_n),e_{\bm u}^{n+1},e_{\bm u}^n)\\
&\le c_0c_2\|e_{\bm u}^n\|_0\|\bm u(t_n)\|_2\|e_{\bm u}^{n+1}\|_1+c_0c_2\|\bm u(t_n)\|_2\|e_{\bm u}^{n+1}\|_1\|e_{\bm u}^n\|_0\\
&\le 2c_0c_2M\|e_{\bm u}^n\|_0\|e_{\bm u}^{n+1}\|_1,
\end{aligned} \label{est:B2_} \vspace{0.4cm}\\
&\begin{aligned}
B_3&=-q^{n+1}b(e_{\bm u}^n,e_{\bm u}^n,e_{\bm u}^{n+1})\le c_0c_2\|e_{\bm u}^n\|_1\|e_{\bm u}^n\|_1\|e_{\bm u}^{n+1}\|_1=c_0c_2\|e_{\bm u}^n\|_1^2\|e_{\bm u}^{n+1}\|_1. \vspace{-0.2cm}
\end{aligned} \label{est:B3_}
\end{align}
We remark that the estimation of $B_3$ involves $\|e_{\bm u}^n\|_1^2$, which eventually becomes $\|e_{\bm u}^n\|_1^4$ after the use of Young's inequality. Alternatively, one may estimate $B_3$ using a stronger inequality $B_3\le c_0c_2\|e_{\bm u}^n\|_0^{1/2}\|e_{\bm u}^n\|_1^{1/2}\|e_{\bm u}^n\|_0^{1/2}\|e_{\bm u}^n\|_1^{1/2}\|e_{\bm u}^{n+1}\|_1$, which simplifies to $B_3\le  c_0c_2\|e_{\bm u}^n\|_0\|e_{\bm u}^n\|_1\|e_{\bm u}^{n+1}\|_1$  and \eqref{est:B3_} follows. Although this estimate avoids the difficulty associated with the quartic term $\|e_{\bm u}^n\|_1^4$, it is only valid in two dimensions ($d=2$), as shown in~\cite{L-Shen22} (where the study is restricted to the 2D setting). To ensure the analysis in the current paper applies to both two and three dimensions, we still use~\eqref{est:B3_} for estimating $B_3$.  Now for $B_4$,  it can be bounded by
\begin{align}\label{est:B4_}
\begin{aligned}
B_4&=-e_q^{n+1}b(\bm u(t_n),\bm u(t_n),e_{\bm u}^{n+1})\le c_2|e_q^{n+1}|\|\bm u(t_n)\|_2\|\bm u(t_n)\|_1\|e_{\bm u}^{n+1}\|_0\le c_2M^2|e_q^{n+1}|\|e_{\bm u}^{n+1}\|_0. 
\end{aligned}
\end{align}
In addition to $B_3$, the estimate for $B_4$ introduces another challenge to the analysis due to the product of implicit error terms, i.e., $|e_q^{n+1}|\|e_{\bm u}^{n+1}\|_0$. We remark that for the first-order SAV scheme in~\cite{L-Shen22}, a term similar to $B_4$ (involving both $e_{\bm u}^{n+1}$ and $e_{SAV}^{n+1}$) can be easily handled by invoking the dynamic equation for the auxiliary variable. However, applying the same approach to $B_4$ in our case leads to suboptimal error estimates for both the velocity and the Lagrange multiplier.  \vspace{0.15cm}

\noindent The combination of \eqref{est:velo}-\eqref{est:B4_} results in 
\begin{align}\label{est:u}
\begin{aligned}
&\frac{\|e_{\bm u}^{n+1}\|_0^2-\|e_{\bm u}^n\|_0^2}{2\tau}+\frac{\|e_{\bm u}^{n+1}-e_{\bm u}^n\|_0^2}{2\tau}+\nu\|\nabla e_{\bm u}^{n+1}\|_0^2\\
&\quad\le \left(\int_{t_n}^{t_{n+1}}\|\bm u_{tt}\|_{-1}\,dt+2c_2M\int_{t_n}^{t_{n+1}}\|\bm u_t\|_0\,dt+2c_0c_2M\|e_{\bm u}^n\|_0+c_0c_2\|e_{\bm u}^n\|_1^2\right)\|e_{\bm u}^{n+1}\|_1\\
&\qquad+c_2M^2|e_q^{n+1}|\|e_{\bm u}^{n+1}\|_0.
\end{aligned}
\end{align}
By the Young's and Cauchy-Schwarz inequalities, noting that $\|e_{\bm u}^{n+1}\|_1\le (1+c_1)\|\nabla e_{\bm u}^{n+1}\|_0$ due to~\eqref{poincare:1}, the first term on the right-hand side of~\eqref{est:u} is bounded by  \vspace{-0.2cm}
\begin{align}\label{est:u_rhs}
\begin{aligned}
&\left(\int_{t_n}^{t_{n+1}}\|\bm u_{tt}\|_{-1}\,dt+2c_2M\int_{t_n}^{t_{n+1}}\|\bm u_t\|_0\,dt+2c_0c_2M\|e_{\bm u}^n\|_0+c_0c_2\|e_{\bm u}^n\|_1^2\right)\|e_{\bm u}^{n+1}\|_1\\
&\quad\le \frac{(1+c_1)^2}{2\nu}\left(\int_{t_n}^{t_{n+1}}\|\bm u_{tt}\|_{-1}\,dt+2c_2M\int_{t_n}^{t_{n+1}}\|\bm u_t\|_0\,dt+2c_0c_2M\|e_{\bm u}^n\|_0+c_0c_2\|e_{\bm u}^n\|_1^2\right)^2\\
&\qquad+\frac{\nu}{2(1+c_1)^2}\|e_{\bm u}^{n+1}\|_1^2\\
&\quad\le \frac{(1+c_1)^2}{2\nu}\left(1+4c_2^2M^2+4c_0^2c_2^2M^2+c_0^2c_2^2\right)\bigg[\left(\int_{t_n}^{t_{n+1}}\|\bm u_{tt}\|_{-1}\,dt\right)^2+\left(\int_{t_n}^{t_{n+1}}\|\bm u_t\|_0\,dt\right)^2\\
&\qquad+\|e_{\bm u}^n\|_0^2+\|e_{\bm u}^n\|_1^4\bigg]+\frac{\nu}{2}\|\nabla e_{\bm u}^{n+1}\|_0^2\\
&\quad\le c_3\left(\tau\int_{t_n}^{t_{n+1}}\|\bm u_{tt}\|_{-1}^2\,dt+\tau\int_{t_n}^{t_{n+1}}\|\bm u_t\|_0^2\,dt+\|e_{\bm u}^n\|_0^2+\|e_{\bm u}^n\|_1^4\right)+\frac{\nu}{2}\|\nabla e_{\bm u}^{n+1}\|_0^2,
\end{aligned}
\end{align}
where
$c_3=\frac{(1+c_1)^2}{2\nu}\left(1+4c_2^2M^2+4c_0^2c_2^2M^2+c_0^2c_2^2\right).$
It follows from~\eqref{est:u} and~\eqref{est:u_rhs} that  \vspace{-0.1cm}
\begin{align}\label{est:u_compact}
\begin{aligned}
&\frac{\|e_{\bm u}^{n+1}\|_0^2-\|e_{\bm u}^n\|_0^2}{2\tau}+\frac{\|e_{\bm u}^{n+1}-e_{\bm u}^n\|_0^2}{2\tau}+\frac{\nu}{2}\|\nabla e_{\bm u}^{n+1}\|_0^2\\
&\quad\le c_3\left(\tau\int_{t_n}^{t_{n+1}}\|\bm u_{tt}\|_{-1}^2\,dt+\tau\int_{t_n}^{t_{n+1}}\|\bm u_t\|_0^2\,dt+\|e_{\bm u}^n\|_0^2+\|e_{\bm u}^n\|_1^4\right)+c_2M^2|e_q^{n+1}|\|e_{\bm u}^{n+1}\|_0.
\end{aligned}
\end{align}
Summing~\eqref{est:u_compact} over $i=0,1,\ldots,n$ and multiplying the resulting estimate by $2\tau$, we deduce that  \vspace{-0.2cm}
\begin{align}\label{u:before_Gronwall}
\begin{aligned}
&\|e_{\bm u}^{n+1}\|_0^2+\sum_{i=0}^n\|e_{\bm u}^{i+1}-e_{\bm u}^i\|_0^2+\tau\nu\sum_{i=0}^n\|\nabla e_{\bm u}^{i+1}\|_0^2\\
&\quad\le 2c_3\tau\left(\tau\int_0^T\|\bm u_{tt}\|_{-1}^2\,dt+\tau\int_0^T\|\bm u_t\|_0^2\,dt+\sum_{i=0}^n\|e_{\bm u}^i\|_0^2+\sum_{i=0}^n\|e_{\bm u}^i\|_1^4\right)\\
&\qquad+2c_2M^2\tau\sum_{i=0}^n |e_q^{i+1}|\|e_{\bm u}^{i+1}\|_0\\
&\quad\le 2c_3\tau\left(2M\tau+\sum_{i=0}^n\|e_{\bm u}^i\|_0^2+\sum_{i=0}^n\|e_{\bm u}^i\|_1^4\right)+2c_2M^2\tau\sum_{i=0}^n |e_q^{i+1}|\|e_{\bm u}^{i+1}\|_0,
\end{aligned}
\end{align}
where we have used the regularity assumption~\eqref{M:velo} in the last estimate. By Young's inequality, the last summation on the right-hand side of~\eqref{u:before_Gronwall} is bounded by
\begin{align}\label{BDT:Cauchy}
\begin{aligned}
&2c_2M^2\tau\sum_{i=0}^n |e_q^{i+1}|\|e_{\bm u}^{i+1}\|_0=2c_2M^2\tau|e_q^{n+1}|\|e_{\bm u}^{n+1}\|_0+2c_2M^2\tau\sum_{i=0}^n |e_q^{i}|\|e_{\bm u}^{i}\|_0\\
&\qquad\le \frac 12\|e_{\bm u}^{n+1}\|_0^2+2c_2^2M^4\tau^2|e_q^{n+1}|^2+c_2M^2\tau\sum_{i=0}^n|e_q^i|^2+c_2M^2\tau\sum_{i=0}^n\|e_{\bm u}^{i}\|_0^2.
\end{aligned}
\end{align}
From~\eqref{u:before_Gronwall} and~\eqref{BDT:Cauchy}, we arrive at
\begin{align}\label{BDT:uq}
\begin{aligned}
&\frac 12\|e_{\bm u}^{n+1}\|_0^2+\sum_{i=0}^n\|e_{\bm u}^{i+1}-e_{\bm u}^i\|_0^2+\tau\nu\sum_{i=0}^n\|\nabla e_{\bm u}^{i+1}\|_0^2\\
&\quad\le 4c_3M\tau^2+(2c_3+c_2M^2)\tau\sum_{i=0}^n\|e_{\bm u}^i\|_0^2+2c_3\tau\sum_{i=0}^n\|e_{\bm u}^i\|_1^4+c_2M^2\tau\sum_{i=0}^n|e_q^i|^2+2c_2^2M^4\tau^2|e_q^{n+1}|^2.
\end{aligned}
\end{align}

\textbf{Step 2: Estimates for the Lagrange multiplier.} Since the right-hand side of~\eqref{BDT:uq} involves the error term $|e_q^{n+1}|^2$ of the Lagrange multiplier, we proceed by estimating this term based on the dynamic equation~\eqref{be3}. This can be achieved by bounding $\left|(q^{n+1})^2-1\right|$, noting that the exact value of $q$ is 1 and $q^0=1$. Observe that~\eqref{be3} can be written as follows
\begin{align}\label{q:info_1}
\theta\frac{(q^{n+1})^2-(q^n)^2}{\tau}&=\frac{\|\bm u^{n+1}-\bm u^n\|_0^2}{2\tau}+q^{n+1}b(\bm u^n,\bm u^n,\bm u^{n+1}).
\end{align}
Taking the $\bm L^2$ inner product of \eqref{be1} with $\bm u^{n+1}-\bm u^n$ and using the properties $b(\bm u^n,\bm u^n,\bm u^n)=0$ and $(\nabla p^{n+1},\bm u^{n+1}-\bm u^n)=0$, we find
\begin{align}\label{q:info_2}
\begin{aligned}
&\frac{\|\bm u^{n+1}-\bm u^n\|_0^2}{\tau}+\frac{\nu}{2}\left(\|\nabla\bm u^{n+1}\|_0^2-\|\nabla\bm u^n\|_0^2+\|\nabla(\bm u^{n+1}-\bm u^n)\|_0^2\right)+q^{n+1}b(\bm u^n,\bm u^n,\bm u^{n+1})\\
&\qquad\qquad=\left(\bm f(t_{n+1}),\bm u^{n+1}-\bm u^n\right).
\end{aligned}
\end{align}
It follows from \eqref{q:info_1} and \eqref{q:info_2} that
\begin{align}\label{q:rewrite}
\begin{aligned}
\theta\frac{(q^{n+1})^2-(q^n)^2}{\tau}=&\left(\bm f(t_{n+1}),\bm u^{n+1}-\bm u^n\right)-\frac{\|\bm u^{n+1}-\bm u^n\|_0^2}{2\tau}\\
&-\frac{\nu}{2}\left(\|\nabla\bm u^{n+1}\|_0^2-\|\nabla\bm u^n\|_0^2+\|\nabla(\bm u^{n+1}-\bm u^n)\|_0^2\right).
\end{aligned}
\end{align}
Notice that $q^0=1$, the sum of \eqref{q:rewrite} over $i=0,1,\ldots,n$ leads to
\begin{align}\label{q:err_eqn}
\begin{aligned}
\theta\frac{(q^{n+1})^2-1}{\tau}=&\sum_{i=0}^n\left(\bm f(t_{i+1}),\bm u^{i+1}-\bm u^i\right)-\frac{1}{2\tau}\sum_{i=0}^n\|\bm u^{i+1}-\bm u^i\|_0^2\\
&-\frac{\nu}{2}(\|\nabla\bm u^{n+1}\|_0^2-\|\nabla\bm u^0\|_0^2)-\frac{\nu}{2}\sum_{i=0}^n\|\nabla(\bm u^{i+1}-\bm u^i)\|_0^2.
\end{aligned}
\end{align}
Recalling $c_{\bm f}=\|\bm f\|_{L^{\infty}(0,T;\bm L^2(\Omega))}$,  the first term on the right-hand side of~\eqref{q:err_eqn} can be bounded by
\begin{align}\label{q:rhs1}
\begin{aligned}
\left|\sum_{i=0}^n\left(\bm f(t_{i+1}),\bm u^{i+1}-\bm u^i\right)\right|&\le \sum_{i=0}^n\|\bm f(t_{i+1})\|_0\|\bm u^{i+1}-\bm u^i\|_0\\
&\le \sum_{i=0}^n c_{\bm f}\sqrt{\tau}\cdot\frac{1}{\sqrt{\tau}}\|\bm u^{i+1}-\bm u^i\|_0\\
&\le \frac 12\sum_{i=0}^nc_{\bm f}^2\tau+\frac{1}{2\tau}\sum_{i=0}^n\|\bm u^{i+1}-\bm u^i\|_0^2\\
&\le \frac 12c_{\bm f}^2T+\frac{1}{2\tau}\sum_{i=0}^n\|\bm u^{i+1}-\bm u^i\|_0^2.
\end{aligned}
\end{align}
Observing that $
\bm u^{i+1}-\bm u^i=(\bm u(t_{i+1})-\bm u(t_i))+(e_{\bm u}^{i+1}-e_{\bm u}^i)$ and using the inequality $\|\bm u+\bm v\|_0^2\le 2(\|\bm u\|_0^2+\|\bm v\|_0^2)$, we can estimate the second term on the right-hand side of~\eqref{q:err_eqn} (as well as~\eqref{q:rhs1}) as follows
\begin{align}\label{q:rhs2}
\begin{aligned}
\frac{1}{2\tau}\sum_{i=0}^n\|\bm u^{i+1}-\bm u^i\|_0^2&\le \frac{1}{\tau}\sum_{i=0}^n\|\bm u(t_{i+1})-\bm u(t_i)\|_0^2+\frac{1}{\tau}\sum_{i=0}^n\|e_{\bm u}^{i+1}-e_{\bm u}^i\|_0^2\\
&\le \frac{1}{\tau}\sum_{i=0}^n\left(\int_{t_i}^{t_{i+1}}\|\bm u_t\|_0\,dt\right)^2+\frac{1}{\tau}\sum_{i=0}^n\|e_{\bm u}^{i+1}-e_{\bm u}^i\|_0^2\\
&\le \int_0^T\|\bm u_t\|_0^2\,dt+\frac{1}{\tau}\sum_{i=0}^n\|e_{\bm u}^{i+1}-e_{\bm u}^i\|_0^2\\
&\le M+\frac{1}{\tau}\sum_{i=0}^n\|e_{\bm u}^{i+1}-e_{\bm u}^i\|_0^2.
\end{aligned}
\end{align}
Similarly, for the term $\frac{\nu}{2}\|\nabla\bm u^{n+1}\|_0^2$ and the last summation on the right-hand side of~\eqref{q:err_eqn}, we have
\begin{align}\label{q:rhs3}
\begin{aligned}
\frac{\nu}{2}\|\nabla\bm u^{n+1}\|_0^2&\le \nu\|\nabla\bm u(t_{n+1})\|_0^2+\nu\|\nabla e_{\bm u}^{n+1}\|_0^2\\
&\le \nu M^2+\nu\|\nabla e_{\bm u}^{n+1}\|_0^2,
\end{aligned}
\end{align}
and
\begin{align}\label{q:rhs4}
\begin{aligned}
\frac{\nu}{2}\sum_{i=0}^n\|\nabla(\bm u^{i+1}-\bm u^i)\|_0^2&\le \nu\sum_{i=0}^n\|\nabla(\bm u(t_{i+1})-\bm u(t_i))\|_0^2+ \nu\sum_{i=0}^n\|\nabla(e_{\bm u}^{i+1}-e_{\bm u}^i)\|_0^2\\
&\le \nu\tau\int_0^T\|\nabla\bm u_t\|_0^2\,dt+\nu\sum_{i=0}^n\|\nabla(e_{\bm u}^{i+1}-e_{\bm u}^i)\|_0^2\\
&\le \nu M\tau+\nu\sum_{i=0}^n\|\nabla(e_{\bm u}^{i+1}-e_{\bm u}^i)\|_0^2.
\end{aligned}
\end{align}
Combining~\eqref{q:err_eqn}-\eqref{q:rhs4}, we obtain
\begin{align}\label{q:end}
\begin{aligned}
\theta\frac{|(q^{n+1})^2-1|}{\tau}&\le \left|\sum_{i=0}^n\left(\bm f(t_{i+1}),\bm u^{i+1}-\bm u^i\right)\right|+ \frac{1}{2\tau}\sum_{i=0}^n\|\bm u^{i+1}-\bm u^i\|_0^2\\
&\qquad+\frac{\nu}{2}\left(\|\nabla\bm u^{n+1}\|_0^2+\|\nabla\bm u^0\|_0^2\right)+\frac{\nu}{2}\sum_{i=0}^n\|\nabla(\bm u^{i+1}-\bm u^i)\|_0^2\\
&\le \frac 12c_{\bm f}^2T+2M+\nu M^2+\frac{\nu}{2}\|\nabla\bm u^0\|_0^2+\nu M\tau\\
&\qquad+\frac{2}{\tau}\sum_{i=0}^n\|e_{\bm u}^{i+1}-e_{\bm u}^i\|_0^2+\nu\|\nabla e_{\bm u}^{n+1}\|_0^2+\nu\sum_{i=0}^n\|\nabla(e_{\bm u}^{i+1}-e_{\bm u}^i)\|_0^2.
\end{aligned}
\end{align}
Notice that $e_q^{n+1}=q^{n+1}-1\ge -1$, which implies
\begin{align}\label{BDT:key}
\begin{aligned}
\theta\frac{|(q^{n+1})^2-1|}{\tau}&=\theta\frac{|(e_q^{n+1}+1)^2-1|}{\tau}\\
&=\theta\frac{|(e_q^{n+1})^2+2e_q^{n+1}|}{\tau}\\
&\ge\theta\frac{\max\{|e_q^{n+1}|^2,|e_q^{n+1}|\}}{\tau}.
\end{aligned}
\end{align}
In the last estimate of~\eqref{BDT:key}, we have used the fact that $|x+2|\ge \max\{|x|,1\}$ for all $x\ge -1$. From~\eqref{q:end} and~\eqref{BDT:key}, we deduce
\begin{align}\label{BDT_max}
\theta\frac{\max\{|e_q^{n+1}|^2,|e_q^{n+1}|\}}{\tau}\le c_4+\nu M\tau+S^{n+1},
\end{align}
where $c_4=\frac 12c_{\bm f}^2T+2M+\nu M^2+\frac{\nu}{2}\|\nabla\bm u^0\|_0^2$ and 
\begin{align*}
S^{n+1}&=\frac{2}{\tau}\sum_{i=0}^n\|e_{\bm u}^{i+1}-e_{\bm u}^i\|_0^2+\nu\|\nabla e_{\bm u}^{n+1}\|_0^2+\nu\sum_{i=0}^n\|\nabla(e_{\bm u}^{i+1}-e_{\bm u}^i)\|_0^2.
\end{align*}

\textbf{Step 3: Estimates for $S^{n+1}$.} 
By taking the $\bm L^2$ inner product of \eqref{err_eq_u} with $e_{\bm u}^{n+1}-e_{\bm u}^n$ and using similar arguments as in the estimates~\eqref{est:Ru_}-\eqref{est:u_rhs}, we obtain
\begin{align*} 
\begin{aligned}
&\frac{\|e_{\bm u}^{n+1}-e_{\bm u}^n\|_0^2}{\tau}+\frac{\nu}{2}\left(\|\nabla e_{\bm u}^{n+1}\|_0^2-\|\nabla e_{\bm u}^n\|_0^2+\|\nabla(e_{\bm u}^{n+1}-e_{\bm u}^n)\|_0^2\right)\\
&\quad=(\bm R_{\bm u}^{n+1},e_{\bm u}^{n+1}-e_{\bm u}^n)+b(\bm u(t_{n+1}),\bm u(t_{n+1}),e_{\bm u}^{n+1}-e_{\bm u}^n)-q^{n+1}b(\bm u^n,\bm u^n,e_{\bm u}^{n+1}-e_{\bm u}^n)\\
&\quad\le\left(\int_{t_n}^{t_{n+1}}\|\bm u_{tt}\|_{-1}\,dt+2c_2M\int_{t_n}^{t_{n+1}}\|\bm u_t\|_0\,dt+2c_0c_2M\|e_{\bm u}^n\|_0+c_0c_2\|e_{\bm u}^n\|_1^2\right)\|e_{\bm u}^{n+1}-e_{\bm u}^n\|_1\\
&\qquad+c_2M^2|e_q^{n+1}|\|e_{\bm u}^{n+1}-e_{\bm u}^n\|_0\\
&\quad\le 2c_3\left(\tau\int_{t_n}^{t_{n+1}}\|\bm u_{tt}\|_{-1}^2\,dt+\tau\int_{t_n}^{t_{n+1}}\|\bm u_t\|_0^2\,dt+\|e_{\bm u}^n\|_0^2+\|e_{\bm u}^n\|_1^4\right)+\frac{\nu}{4}\|\nabla(e_{\bm u}^{n+1}-e_{\bm u}^n)\|_0^2\\
&\qquad+\frac{\|e_{\bm u}^{n+1}-e_{\bm u}^n\|_0^2}{2\tau}+\frac{c_2^2M^4\tau}{2}|e_q^{n+1}|^2,
\end{aligned}
\end{align*}
which implies that
\begin{align}\label{u:other}
\begin{aligned}
&\frac{\|e_{\bm u}^{n+1}-e_{\bm u}^n\|_0^2}{2\tau}+\frac{\nu}{2}\left(\|\nabla e_{\bm u}^{n+1}\|_0^2-\|\nabla e_{\bm u}^n\|_0^2\right)+\frac{\nu}{4}\|\nabla(e_{\bm u}^{n+1}-e_{\bm u}^n)\|_0^2\\
&\quad\le 2c_3\left(\tau\int_{t_n}^{t_{n+1}}\|\bm u_{tt}\|_{-1}^2\,dt+\tau\int_{t_n}^{t_{n+1}}\|\bm u_t\|_0^2\,dt+\|e_{\bm u}^n\|_0^2+\|e_{\bm u}^n\|_1^4\right)+\frac{c_2^2M^4\tau}{2}|e_q^{n+1}|^2.
\end{aligned}
\end{align}
The sum of \eqref{u:other} over $i=0,1,\ldots,n$ gives us
\begin{align}\label{u:key1}
\begin{aligned}
&\frac{1}{2\tau}\sum_{i=0}^n\|e_{\bm u}^{i+1}-e_{\bm u}^i\|_0^2+\frac{\nu}{2}\|\nabla e_{\bm u}^{n+1}\|_0^2+\frac{\nu}{4}\sum_{i=0}^n\|\nabla(e_{\bm u}^{i+1}-e_{\bm u}^i)\|_0^2\\
&\quad\le 2c_3\left(2M\tau+\sum_{i=0}^n\|e_{\bm u}^i\|_0^2+\sum_{i=0}^n\|e_{\bm u}^i\|_1^4\right)+\frac{c_2^2M^4\tau}{2}\sum_{i=0}^n|e_q^{i+1}|^2.
\end{aligned}
\end{align}
By multiplying both sides of~\eqref{u:key1} by 4, we can bound $S^{n+1}$ as follows
\begin{align}\label{u:key2}
\begin{aligned}
S^{n+1}&\le\frac{2}{\tau}\sum_{i=0}^n\|e_{\bm u}^{i+1}-e_{\bm u}^i\|_0^2+2\nu\|\nabla e_{\bm u}^{n+1}\|_0^2+\nu\sum_{i=0}^n\|\nabla(e_{\bm u}^{i+1}-e_{\bm u}^i)\|_0^2\\
&\le 8c_3\left(2M\tau+\sum_{i=0}^n\|e_{\bm u}^i\|_0^2+\sum_{i=0}^n\|e_{\bm u}^i\|_1^4\right)+2c_2^2M^4\tau\sum_{i=0}^n|e_q^{i+1}|^2.
\end{aligned}
\end{align}

\textbf{Step 4: Simultaneous estimates for the velocity and Lagrange multiplier.} In light of Step 2 and Step 3, let us first simplify the estimate for $\theta|e_q^{n+1}|^2$. From~\eqref{BDT_max} and~\eqref{u:key2}, we find
\begin{align}\label{for_later}
\begin{aligned}
&\theta\frac{\max\{|e_q^{n+1}|^2,|e_q^{n+1}|\}}{\tau}\le c_4+\nu M\tau+S^{n+1}\\
&\quad\le c_4+(16c_3+\nu)M\tau+ 8c_3\sum_{i=0}^n\|e_{\bm u}^i\|_0^2+8c_3\sum_{i=0}^n\|e_{\bm u}^i\|_1^4+2c_2^2M^4\tau\sum_{i=0}^n|e_q^{i+1}|^2.
\end{aligned}
\end{align}
Since $\frac{x+y}{2}\le \max\{x,y\}$ for all $x,y\in\mathbb{R}$, multiplying~\eqref{for_later} by $2\tau$ yields
\begin{align}\label{q:c3c4c5}
\begin{aligned}
\theta|e_q^{n+1}|^2+\theta|e_q^{n+1}|&\le 2c_4\tau+(32c_3+2\nu)M\tau^2+ 16c_3\tau\sum_{i=0}^n\|e_{\bm u}^i\|_0^2\\
&\quad+16c_3\tau\sum_{i=0}^n\|e_{\bm u}^i\|_1^4+4c_2^2M^4\tau^2\sum_{i=0}^n|e_q^{i+1}|^2.
\end{aligned}
\end{align}
If $\theta|e_q^{n+1}|>2c_4\tau$ then~\eqref{q:c3c4c5} reduces to 
\begin{align*}
\theta|e_q^{n+1}|^2\le (32c_3+2\nu)M\tau^2+ 16c_3\tau\sum_{i=0}^n\|e_{\bm u}^i\|_0^2+16c_3\tau\sum_{i=0}^n\|e_{\bm u}^i\|_1^4+4c_2^2M^4\tau^2\sum_{i=0}^n|e_q^{i+1}|^2.
\end{align*}
If $\theta|e_q^{n+1}|\le 2c_4\tau$ then $\theta|e_q^{n+1}|^2\le \frac{4c_4^2}{\theta}\tau^2$. In both cases, we have  \vspace{-0.2cm}
\begin{align}\label{BDT:qu}
\begin{aligned}
\theta|e_q^{n+1}|^2\le \left(32c_3M+2\nu M+\frac{4c_4^2}{\theta}\right)\tau^2+ 16c_3\tau\sum_{i=0}^n\|e_{\bm u}^i\|_0^2
+16c_3\tau\sum_{i=0}^n\|e_{\bm u}^i\|_1^4+4c_2^2M^4\tau^2\sum_{i=0}^n|e_q^{i+1}|^2.
\end{aligned}
\end{align}
Taking the sum of~\eqref{BDT:uq} and~\eqref{BDT:qu}, we obtain  \vspace{-0.2cm}
\begin{align}\label{combine_uq}
\begin{aligned}
&\frac 12\|e_{\bm u}^{n+1}\|_0^2+\sum_{i=0}^n\|e_{\bm u}^{i+1}-e_{\bm u}^i\|_0^2+\tau\nu\sum_{i=0}^n\|\nabla e_{\bm u}^{i+1}\|_0^2+\theta|e_q^{n+1}|^2\\
&\quad\le \left(36c_3M+2\nu M+\frac{4c_4^2}{\theta}\right)\tau^2+(18c_3+c_2M^2)\tau\sum_{i=0}^n\|e_{\bm u}^i\|_0^2+18c_3\tau\sum_{i=0}^n\|e_{\bm u}^i\|_1^4\\
&\qquad+(c_2M^2+4c_2^2M^4\tau)\tau\sum_{i=0}^n|e_q^i|^2+6c_2^2M^4\tau^2|e_q^{n+1}|^2.
\end{aligned}
\end{align}
If $\tau$ satisfies $6c_2^2M^4\tau^2\le \frac{\theta}{2}$, or   \vspace{-0.2cm}
\begin{align}\label{CFL}
 \tau\le \frac{\sqrt{\theta}}{2\sqrt{3}c_2M^2},  
\end{align}
then multiplying~\eqref{combine_uq} by 2 yields  \vspace{-0.2cm}
\begin{align}\label{combine_uq_new}
\begin{aligned}
&\|e_{\bm u}^{n+1}\|_0^2+2\sum_{i=0}^n\|e_{\bm u}^{i+1}-e_{\bm u}^i\|_0^2+2\tau\nu\sum_{i=0}^n\|\nabla e_{\bm u}^{i+1}\|_0^2+\theta|e_q^{n+1}|^2\\
&\quad\le \left(72c_3M+4\nu M+\frac{8c_4^2}{\theta}\right)\tau^2+(36c_3+2c_2M^2)\tau\sum_{i=0}^n\|e_{\bm u}^i\|_0^2+36c_3\tau\sum_{i=0}^n\|e_{\bm u}^i\|_1^4\\
&\qquad+(2c_2M^2+8c_2^2M^4\tau)\tau\sum_{i=0}^n|e_q^i|^2\\
&\quad\le \max\left\{36c_3+2c_2M^2,\frac{2c_2M^2+8c_2^2M^4\tau}{\theta}\right\}\tau\sum_{i=0}^n\left(\|e_{\bm u}^{i}\|_0^2+\theta|e_q^i|^2\right)\\
&\qquad +\left(72c_3M+4\nu M+\frac{8c_4^2}{\theta}\right)\tau^2+36c_3\tau\sum_{i=0}^n\|e_{\bm u}^i\|_1^4.
\end{aligned}
\end{align}
From \eqref{CFL} and the fact $\theta\ge 1$, we know that  \vspace{-0.2cm}
\begin{align*}
\frac{2c_2M^2+8c_2^2M^4\tau}{\theta}&=2c_2M^2\frac{1+4c_2M^2\tau}{\theta}\le 2c_2M^2\frac{1+\frac{2\sqrt{\theta}}{\sqrt{3}}}{\theta}\le 2c_2M^2\left(1+\frac{2}{\sqrt{3}}\right)\le 5c_2M^2.
\end{align*}
Therefore, \eqref{combine_uq_new} can be simplified as  \vspace{-0.2cm}
\begin{align}\label{combine_uq_new!}
\begin{aligned}
&\|e_{\bm u}^{n+1}\|_0^2+2\sum_{i=0}^n\|e_{\bm u}^{i+1}-e_{\bm u}^i\|_0^2+2\tau\nu\sum_{i=0}^n\|\nabla e_{\bm u}^{i+1}\|_0^2+\theta|e_q^{n+1}|^2\\
&\le (36c_3+5c_2M^2)\tau\sum_{i=0}^n\left(\|e_{\bm u}^{i}\|_0^2+\theta|e_q^i|^2\right)+\left(72c_3M+4\nu M+\frac{8c_4^2}{\theta}\right)\tau^2+36c_3\tau\sum_{i=0}^n\|e_{\bm u}^i\|_1^4.
\end{aligned}
\end{align}
Applying the discrete Gronwall's inequality in Lemma~\ref{lemma:Gronwall} to~\eqref{combine_uq_new!}, with $\tilde a_n=\|e_{\bm u}^{n}\|_0^2+\theta|e_q^n|^2$ for $0\le n\le N$, we obtain
\begin{align}\label{result_1}
\begin{aligned}
&\|e_{\bm u}^{n+1}\|_0^2+2\sum_{i=0}^n\|e_{\bm u}^{i+1}-e_{\bm u}^i\|_0^2+2\tau\nu\sum_{i=0}^n\|\nabla e_{\bm u}^{i+1}\|_0^2+\theta|e_q^{n+1}|^2\\
&\quad\le \left[\left(72c_3M+4\nu M+\frac{8c_4^2}{\theta}\right)\tau^2+36c_3\tau\sum_{i=0}^n\|e_{\bm u}^i\|_1^4\right]\mathrm{exp}\left(T(36c_3+5c_2M^2)\right)\\
&\quad\le\left(2c_5M+\frac{c_6}{\theta}\right)\tau^2+c_5\tau\sum_{i=0}^n\|e_{\bm u}^i\|_1^4,
\end{aligned}
\end{align}
where $c_5=\left(36c_3+2\nu\right)\mathrm{exp}\left(T(36c_3+5c_2M^2)\right)$ and  $c_6=8c_4^2\,\mathrm{exp}\left(T(36c_3+5c_2M^2)\right)$.

\textbf{Step 5: Estimates for $\sum\limits_{i=0}^n\|e_{\bm u}^i\|_1^4$ by induction.}
It follows from~\eqref{result_1} and Poincaré inequality~\eqref{poincare:1} that
\begin{align*}
\frac{2\tau\nu}{(1+c_1)^2}\sum_{i=0}^{n}\|e_{\bm u}^{i+1}\|_1^2&\le 2\tau\nu\sum_{i=0}^n\|\nabla e_{\bm u}^{i+1}\|_0^2\le \left(2c_5M+\frac{c_6}{\theta}\right)\tau^2+c_5\tau\left(\sum_{i=0}^n\|e_{\bm u}^i\|_1^2\right)^2,
\end{align*}
which leads to
\begin{align}\label{for_induction}
\sum_{i=0}^{n+1}\|e_{\bm u}^i\|_1^2\le \frac{(1+c_1)^2}{2\nu}\left[2c_5M+\frac{c_6}{\theta}+\frac{c_5}{\tau}\left(\sum_{i=0}^n\|e_{\bm u}^i\|_1^2\right)^2\right]\tau.
\end{align}
Using~\eqref{for_induction}, we will prove by induction that, for sufficiently small $\tau$, the following holds for $0\le n\le N$:
\begin{align}\label{induction}
\sum_{i=0}^n\|e_{\bm u}^i\|_1^2\le \frac{(1+c_1)^2}{2\nu}\left(2c_5M+\frac{c_6}{\theta}+c_5\right)\tau=:\alpha\tau.
\end{align}
Clearly, \eqref{induction} holds for $n=0$ since $e_{\bm u}^0=\bm 0$. Suppose it is true for some nonnegative integer $n<N$. In light of~\eqref{for_induction}, for $\tau$ satisfying 
\begin{align}\label{CFL_2}
 \tau\le\frac{1}{\alpha^2}=\frac{4\nu^2}{(1+c_1)^4\left(c_5+2c_5M+\dfrac{c_6}{\theta}\right)^2},
\end{align}
we find that
\begin{align*}
\sum_{i=0}^{n+1}\|e_{\bm u}^i\|_1^2&\le \frac{(1+c_1)^2}{2\nu}\left(2c_5M+\frac{c_6}{\theta}+c_5\alpha^2\tau\right)\tau\le \frac{(1+c_1)^2}{2\nu}\left(2c_5M+\frac{c_6}{\theta}+c_5\right)\tau=\alpha\tau,
\end{align*}
which verifies the statement~\eqref{induction}, provided that $\tau\le\frac{1}{\alpha^2}$. Therefore, we have the following estimate
\begin{align}\label{result_2}
\sum_{i=0}^N\|e_{\bm u}^i\|_1^4\le \left(\sum_{i=0}^N\|e_{\bm u}^i\|_1^2\right)^2\le \alpha^2\tau^2\le \tau.
\end{align}

\textbf{Step 6: Optimal convergence rates for the velocity and Lagrange multiplier.}
Combining~\eqref{result_1} and~\eqref{result_2}, we end up with 
\begin{align}\label{result_3}
\begin{aligned}
&\|e_{\bm u}^{n+1}\|_0^2+2\sum_{i=0}^n\|e_{\bm u}^{i+1}-e_{\bm u}^i\|_0^2+2\tau\nu\sum_{i=0}^n\|\nabla e_{\bm u}^{i+1}\|_0^2+\theta|e_q^{n+1}|^2\le \left(c_5+2c_5M+\frac{c_6}{\theta}\right)\tau^2,
\end{aligned}
\end{align}
provided that $\tau$ satisfies the following condition (cf. \eqref{CFL} and~\eqref{CFL_2} and the fact $\theta\ge 1$)
$$\tau\le \min\left\{\frac{1}{2\sqrt{3}c_2M^2},\frac{4\nu^2}{(1+c_1)^4\left(c_5+2c_5M+c_6\right)^2},T\right\}=:\tau_0.$$
Setting $c_7=\max\{c_5+2c_5M,c_6\}$, we obtain from~\eqref{result_3} that for $0\le n<N$,
\begin{align}
\|e_{\bm u}^{n+1}\|_0^2+2\sum_{i=0}^n\|e_{\bm u}^{i+1}-e_{\bm u}^i\|_0^2+2\tau\nu\sum_{i=0}^n\|\nabla e_{\bm u}^{i+1}\|_0^2&\le c_7\left(1+\frac{1}{\theta}\right)\tau^2, \label{result_4}  \vspace{-0.2cm}
\end{align}
and  \vspace{-0.2cm}
\begin{align}
|e_q^{n+1}|^2&\le c_7\left(\frac{1}{\theta}+\frac{1}{\theta^2}\right)\tau^2.\label{BDT_improve}  \vspace{-0.2cm}
\end{align}
In light of~\eqref{result_2},~\eqref{result_4}, and~\eqref{BDT_improve}, the bound for $|e_q^{n+1}|$ can be further improved (especially when $\theta\gg 1$) using~\eqref{for_later} as follows  \vspace{-0.2cm}
\begin{align*}
\theta\frac{|e_q^{n+1}|}{\tau}&\le c_4+(16c_3+\nu)M\tau+ 8c_3\sum_{i=0}^n\|e_{\bm u}^i\|_0^2+8c_3\sum_{i=0}^n\|e_{\bm u}^i\|_1^4+2c_2^2M^4\tau\sum_{i=0}^n|e_q^{i+1}|^2\\
&\le c_4+(16c_3+\nu)M\tau+8c_3c_7\left(1+\frac{1}{\theta}\right)T\tau+8c_3\tau+2c_2^2M^4c_7\left(\frac{1}{\theta}+\frac{1}{\theta^2}\right)T\tau^2\\
&\le c_4+c_8\tau+c_9\tau^2,
\end{align*}
where $c_8=(16c_3+\nu)M+16c_3c_7T+8c_3$ and $c_9=4c_2^2M^4c_7T$. This means that
$$|e_q^{n+1}|\le \frac{c_4+c_8\tau+c_9\tau^2}{\theta}\tau,$$
and the proof of Theorem~\ref{thm:DRLM1} is completed with $C_1=c_7$ and $C_2=c_4+c_8\tau_0+c_9\tau_0^2$.
\end{proof}

\begin{remark} It is necessary to follow Steps 2-6 in the proof above to bound the Lagrange multiplier error $e_q^{n+1}$ from both above and below,  especially the latter.  Note that an upper bound for $e_q^{n+1}$ can be easily obtained by combining~\eqref{q:err_eqn} and~\eqref{q:rhs1} which yields
\begin{align}\label{rm:upper_bound}
\begin{aligned}
\theta\frac{(q^{n+1})^2-1}{\tau}&\le \frac 12c_{\bm f}^2T+\frac{\nu}{2}\|\nabla\bm u^0\|_0^2-\frac{\nu}{2}\|\nabla\bm u^{n+1}\|_0^2-\frac{\nu}{2}\sum_{i=0}^n\|\nabla(\bm u^{i+1}-\bm u^i)\|_0^2\\
&\le \frac 12c_{\bm f}^2T+\frac{\nu}{2}\|\nabla\bm u^0\|_0^2=:\widehat{c}.
\end{aligned}
\end{align}
Recall that $q^{n+1}-1=e_q^{n+1}$ and $q^{n+1}+1>1$,~\eqref{rm:upper_bound} implies 
$$\theta\frac{e_q^{n+1}}{\tau}\le \widehat{c},\text{ or equivalently, } e_q^{n+1}\le \frac{\widehat{c}}{\theta}\tau.$$
\end{remark}

\subsection{Optimal error estimates for the pressure}
We begin with a key lemma on the discrete derivative of the Lagrange multiplier error, which is crucial for the pressure error estimate.  For any sequence $\{g^n\}_{n\ge 0}$, let us define
\begin{align*}
d_tg^{n+1}:=\frac{g^{n+1}-g^n}{\tau}, \quad \text{for any} \; n\ge 0.
\end{align*}
\begin{lemma}\label{lemma:dteq}
Under the regularity assumptions~\eqref{regularity:must} and $\bm u_t\in L^4(0,T;\bm L^2(\Omega))$, there exists a positive constant $C_3$ independent of $\tau$ and $\theta$ such that for all $\tau\le \min\{\tau_0,\frac{\theta}{2C_2}\}$, we have
\begin{align*}
\sum_{n=0}^{N-1}|d_te_q^{n+1}|^2\le \frac{C_3}{\theta^2}\tau,
\end{align*}
where the constants $\tau_0$ and $C_2$ are defined in Theorem~\ref{thm:DRLM1}.
\end{lemma}

\begin{proof}\hspace{0.1cm} 
Let $\widetilde{M}$ be a constant such that $\displaystyle\int_{0}^{T}\|\bm u_t\|_0^4\,dt\le \widetilde{M}$.
Since $q^{n+1}-q^n=e_q^{n+1}-e_q^n$, it follows from~\eqref{q:rewrite} that
\begin{align}\label{dt:eq}
\theta(q^{n+1}+q^n)d_te_q^{n+1}=\left(\bm f(t_{n+1}),\bm u^{n+1}-\bm u^n\right)-\frac{\|\bm u^{n+1}-\bm u^n\|_0^2}{2\tau}-\nu(\nabla\bm u^{n+1},\nabla(\bm u^{n+1}-\bm u^n)).
\end{align}
Recalling that $c_{\bm f}=\|\bm f\|_{L^{\infty}(0,T;\bm L^2(\Omega))}$,
the first term on the right-hand side of~\eqref{dt:eq} is bounded by
\begin{align}\label{est:cf}
\begin{aligned}
\left|\left(\bm f(t_{n+1}),\bm u^{n+1}-\bm u^n\right)\right|&\le \|\bm f(t_{n+1})\|_0\|\bm u^{n+1}-\bm u^n\|_0\\
&\le c_{\bm f}\|\bm u^{n+1}-\bm u^n\|_0\\
&\le \frac 12c_{\bm f}^2\tau+\frac{1}{2\tau}\|\bm u^{n+1}-\bm u^n\|_0^2.
\end{aligned}
\end{align}
It is implied from~\eqref{dt:eq} and~\eqref{est:cf} that
\begin{align}\label{C3hat:1}
\theta(q^{n+1}+q^n)|d_te_q^{n+1}|\le \frac 12c_{\bm f}^2\tau+\frac{1}{\tau}\|\bm u^{n+1}-\bm u^n\|_0^2+\nu\|\nabla\bm u^{n+1}\|_0\|\nabla(\bm u^{n+1}-\bm u^n)\|_0.
\end{align}
To estimate the right-hand side of~\eqref{C3hat:1}, we observe from Theorem~\ref{thm:DRLM1} that
\begin{align}\label{C3hat:2}
\begin{aligned}
\|\bm u^{n+1}-\bm u^n\|_0^2&\le \left(\|e_{\bm u}^{n+1}-e_{\bm u}^{n}\|_0+\|\bm u(t_{n+1})-\bm u(t_n)\|_0\right)^2\\
&\le \left(\sqrt{2C_1}\,\tau+\int_{t_n}^{t_{n+1}}\|\bm u_t\|_0\,dt\right)^2\\
&\le (2C_1+1)\left(\tau^2+\tau\int_{t_n}^{t_{n+1}}\|\bm u_t\|_0^2\,dt\right),\\
\|\nabla\bm u^{n+1}\|_0&\le \|\nabla e_{\bm u}^{n+1}\|_0+\|\nabla\bm u(t_{n+1})\|_0\\
&\le \sqrt{\frac{2C_1\tau}{\nu}}+M\le \sqrt{\frac{2C_1T}{\nu}}+M,\\
\|\nabla(\bm u^{n+1}-\bm u^n)\|_0&\le \|\nabla(e_{\bm u}^{n+1}-e_{\bm u}^n)\|_0+\|\nabla(\bm u(t_{n+1})-\bm u(t_n))\|_0\\
&\le \|\nabla e_{\bm u}^{n+1}\|_0+\|\nabla e_{\bm u}^n\|_0+\int_{t_n}^{t_{n+1}}\|\bm u_t\|_1\,dt.
\end{aligned}
\end{align}
When $\tau\le \frac{\theta}{2C_2}$, we have $|e_q^n|\le \frac 12$ for $0\le n\le N$. Therefore, it holds $\frac 12\le q^n\le \frac 32$ for $0\le n\le N$, and thus $q^{n+1}+q^n\ge 1$. Collecting~\eqref{C3hat:1} and~\eqref{C3hat:2} yields
\begin{align}\label{pre_C3}
\begin{aligned}
\theta|d_te_q^{n+1}|&\le \theta(q^{n+1}+q^n)|d_te_q^{n+1}|\\
&\le \left(\frac 12c_{\bm f}^2+2C_1+1\right)\tau+(2C_1+1)\int_{t_n}^{t_{n+1}}\|\bm u_t\|_0^2\,dt\\
&\quad+\left(\sqrt{2C_1T\nu}+M\nu\right)\left(\|\nabla e_{\bm u}^{n+1}\|_0+\|\nabla e_{\bm u}^n\|_0+\int_{t_n}^{t_{n+1}}\|\bm u_t\|_1\,dt\right).
\end{aligned}
\end{align}
Squaring both sides of~\eqref{pre_C3} and applying the Cauchy-Schwarz inequality, we find
\begin{align}\label{theta**2}
\begin{aligned}
\theta^2|d_te_q^{n+1}|^2&\le \left[\left(\frac 12c_{\bm f}^2+2C_1+1\right)^2+(2C_1+1)^2+3\left(\sqrt{2C_1T\nu}+M\nu\right)^2\right]\\
&\quad\cdot\left[\tau^2+\tau\int_{t_n}^{t_{n+1}}\|\bm u_t\|_0^4\,dt+\|\nabla e_{\bm u}^{n+1}\|_0^2+\|\nabla e_{\bm u}^n\|_0^2+\tau\int_{t_n}^{t_{n+1}}\|\bm u_t\|_1^2\,dt\right].
\end{aligned}
\end{align}
By the regularity of $\bm u$ and Theorem~\ref{thm:DRLM1}, we have 
\begin{align}\label{M_nu}
\int_{0}^{T}\|\bm u_t\|_0^4\,dt\le \widetilde{M},\quad \int_0^T\|\bm u_t\|_1^2\,dt\le M,\quad \sum_{n=0}^N\|\nabla e_{\bm u}^n\|_0^2\le \frac{2C_1}{\nu}\tau.
\end{align}
The proof of Lemma~\ref{lemma:dteq} is completed by summing~\eqref{theta**2} over $n=0,1,\ldots,N-1$ and using~\eqref{M_nu}, where the resulting constant $C_3$, given by
\begin{align*}
C_3=\left[\left(\frac 12c_{\bm f}^2+2C_1+1\right)^2+(2C_1+1)^2+3\left(\sqrt{2C_1T\nu}+M\nu\right)^2\right]\left(T+\widetilde{M}+M+\frac{4C_1}{\nu}\right),
\end{align*}
is independent of $\tau$ and $\theta$.
\end{proof}

Next, we state the main result in this subsection and present a detailed proof of the pressure error estimate. In addition to~\eqref{regularity:must}, the following extra regularity assumptions are imposed (cf. Remark~\ref{rm:pressure}):
\begin{align}\label{regularity:extra}
\begin{aligned}
&\bm u_t\in L^2(0,T;\bm H^2(\Omega))\cap L^4(0,T;\bm H^1(\Omega))\cap L^{\infty}(0,\tau_0;\bm H^1(\Omega)),\\
&\bm u_{tt}\in L^2(0,T;\bm L^2(\Omega))\cap L^{\infty}(0,\tau_0;\bm L^2(\Omega)),\quad \bm u_{ttt}\in L^2(0,T;\bm H^{-1}(\Omega)).
\end{aligned}
\end{align}
Let $\widehat{M}$ be a constant such that
\begin{align}\label{M:pressure}
\begin{aligned}
\max\left\{\sup_{t\in [0,\tau_0]}\|\bm u_t\|_1,\;\int_0^T\|\bm u_t\|_2^2\,dt,\;\int_0^T\|\bm u_t\|_1^4\,dt\right\}\le \widehat{M},\\
\max\left\{\sup_{t\in [0,\tau_0]}\|\bm u_{tt}\|_0,\;\int_0^T\|\bm u_{tt}\|_0^2\,dt,\;\int_0^T\|\bm u_{ttt}\|_{-1}^2\,dt\right\}\le \widehat{M}.
\end{aligned}
\end{align}
\begin{theorem}\label{thm:pressure}
Let $\{\bm u^n\}$, $\{p^n\}$, and $\{q^n\}$ be generated by the first-order DRLM scheme~\eqref{DRLM1}. There exist positive constants $\widehat{\tau}_0$ and $C_4$ depending on $\Omega$, $T$, $\nu$, $M$, $\widehat{M}$, $\bm u_0$, and $\bm f$ but independent of $\tau$, $\theta$, and $n$ such that the following error estimate holds for all $\tau\le \min\{\tau_0,\widehat{\tau}_0,\frac{\theta}{2C_2}\}$: 
\begin{align*}
{\tau} \sum_{i=0}^n\|e_p^{i+1}\|^2_{L^2/\mathbb R}\le C_4\left(1+\frac{1}{\theta}\right)\tau^{2}, \quad 0\le n\le N-1,
\end{align*}
where the constants $\tau_0$ and $C_2$ are given in Theorem~\ref{thm:DRLM1}. 
\end{theorem}

\begin{proof}\hspace{0.1cm}
To obtain optimal error estimates for the pressure, we first apply the inf-sup condition to derive an upper bound for the pressure error that involves only the velocity and Lagrange multiplier. This leads to estimating the discrete time derivative of the velocity error -- the main step in the proof, which can be carried out using Lemma~\ref{lemma:dteq} and the regularity assumptions on~$\bm u$. Finally, the desired estimate is obtained using Gronwall's inequality in Lemma~\ref{lemma:Gronwall}. 

\textbf{Step 1: Estimates for the pressure.} A standard approach to estimating the pressure error is to apply the inf-sup (or LBB) condition after establishing the velocity error bound (cf.~\cite{L-Shen20,L-Shen22}). Equivalently, we note from~\cite[Lemma IV.1.9]{Boyer13} that there exists a constant $\beta>0$ such that for all $\varphi\in L^2(\Omega)/\mathbb R$, we have
\begin{align}\label{pr:switch}
\frac{1}{\beta}\|\nabla\varphi\|_{-1}\le\|\varphi\|_{L^2/\mathbb R}\le \beta\|\nabla\varphi\|_{-1}.    
\end{align}
In order to bound $\|e_p^n\|_{L^2/\mathbb R}$ for $1\le n\le N$, it suffices to estimate $\|\nabla e_p^n\|_{-1}$. From~\eqref{err_eq_u}, we have
\begin{align}\label{eq:pressure}
\begin{aligned}
\nabla e_p^{n+1}=-\frac{e_{\bm u}^{n+1}-e_{\bm u}^n}{\tau}+\nu\Delta e_{\bm u}^{n+1}+\bm R_{\bm u}^{n+1}+(\bm u(t_{n+1})\cdot\nabla)\bm u(t_{n+1})-q^{n+1}(\bm u^n\cdot\nabla)\bm u^n.
\end{aligned}
\end{align}
For any $\bm v\in\bm H_0^1(\Omega)$, the $\bm L^2$ inner product of~\eqref{eq:pressure} with $\bm v$ leads to
\begin{align}\label{eq:any_v}
\begin{aligned}
\left(\nabla e_p^{n+1},\bm v\right)&=-\frac{1}{\tau}\left(e_{\bm u}^{n+1}-e_{\bm u}^n,\bm v\right)-\nu\left(\nabla e_{\bm u}^{n+1},\nabla\bm v\right)+\left(\bm R_{\bm u}^{n+1},\bm v\right)\\
&\qquad+b\left(\bm u(t_{n+1}),\bm u(t_{n+1}),\bm v\right)-q^{n+1}b\left(\bm u^n,\bm u^n,\bm v\right)\\
&\le \frac{1}{\tau}\|e_{\bm u}^{n+1}-e_{\bm u}^n\|_{-1}\|\bm v\|_1+\nu\|\nabla e_{\bm u}^{n+1}\|_0\|\bm v\|_1+\bigg(\int_{t_n}^{t_{n+1}}\|\bm u_{tt}\|_{-1}\,dt\\
&\quad +2c_2M\int_{t_n}^{t_{n+1}}\|\bm u_t\|_0\,dt+2c_0c_2M\|e_{\bm u}^n\|_0+c_0c_2\|e_{\bm u}^n\|_1^2+c_2M^2|e_q^{n+1}|\bigg)\|\bm v\|_1,
\end{aligned}
\end{align}
where in the last inequality, we have used estimates~\eqref{est:Ru_}-\eqref{est:B4_} with $e_{\bm u}^{n+1}$ replaced by $\bm v$. By the definition of $\bm H^{-1}$-norm, we know that
$$\|\nabla e_p^{n+1}\|_{-1}=\inf\left\{c\ge 0:\left(\nabla e_p^{n+1},\bm v\right)\le c\|\bm v\|_1,\;\forall\,\bm v\in\bm H_0^1(\Omega)\right\}.$$
Therefore, \eqref{eq:any_v} gives us
\begin{align}\label{est:-1norm}
\begin{aligned}
\|\nabla e_p^{n+1}\|_{-1}&\le \frac{1}{\tau}\|e_{\bm u}^{n+1}-e_{\bm u}^n\|_{-1}+\nu\|\nabla e_{\bm u}^{n+1}\|_0+\int_{t_n}^{t_{n+1}}\|\bm u_{tt}\|_{-1}\,dt\\
&\quad+2c_2M\int_{t_n}^{t_{n+1}}\|\bm u_t\|_0\,dt+2c_0c_2M\|e_{\bm u}^n\|_0+c_0c_2\|e_{\bm u}^n\|_1^2+c_2M^2|e_q^{n+1}|.
\end{aligned}
\end{align}
Squaring both sides of~\eqref{est:-1norm} and applying the Cauchy-Schwarz inequality, we deduce that
\begin{align}\label{est:x}
\begin{aligned}
\|\nabla e_p^{n+1}\|_{-1}^2&\le \left(1+\nu+1+4c_2^2M^2+4c_0^2c_2^2M^2+c_0^2c_2^2+c_2^2M^4\right)\\
&\quad\cdot\bigg(\frac{1}{\tau^2}\|e_{\bm u}^{n+1}-e_{\bm u}^n\|_{-1}^2+\nu\|\nabla e_{\bm u}^{n+1}\|_0^2+\tau\int_{t_n}^{t_{n+1}}\|\bm u_{tt}\|_{-1}^2\,dt\\
&\qquad +\tau\int_{t_n}^{t_{n+1}}\|\bm u_t\|_0^2\,dt+\|e_{\bm u}^n\|_0^2+\|e_{\bm u}^n\|_1^4+|e_q^{n+1}|^2\bigg).
\end{aligned}
\end{align}
Let $c_{10}=2+\nu+4c_2^2M^2+4c_0^2c_2^2M^2+c_0^2c_2^2+c_2^2M^4$. The sum of~\eqref{est:x} over $i=0,1,\ldots,n$ yields
\begin{align}\label{est:c9}
\begin{aligned}
\sum_{i=0}^n\|\nabla e_p^{i+1}\|_{-1}^2&\le c_{10}\bigg(\frac{1}{\tau^2}\sum_{i=0}^{n}\|e_{\bm u}^{i+1}-e_{\bm u}^i\|_{-1}^2+\nu\sum_{i=0}^{n}\|\nabla e_{\bm u}^{i+1}\|_0^2+\tau\int_0^T\|\bm u_{tt}\|_{-1}^2\,dt\\
&\qquad\quad+\tau\int_{0}^{T}\|\bm u_t\|_0^2\,dt+\sum_{i=0}^{n}\|e_{\bm u}^i\|_0^2+\sum_{i=0}^{n}\|e_{\bm u}^i\|_1^4+\sum_{i=0}^{n}|e_q^{i+1}|^2\bigg).
\end{aligned}
\end{align}
By using Theorem~\ref{thm:DRLM1}, the regularity assumption~\eqref{M:velo}, and the estimate~\eqref{result_2}, we obtain from~\eqref{est:c9} that
\begin{align}\label{pr:-1}
\begin{aligned}
\sum_{i=0}^n\|\nabla e_p^{i+1}\|_{-1}^2&\le c_{10}\bigg[\frac{1}{\tau^2}\sum_{i=0}^{n}\|e_{\bm u}^{i+1}-e_{\bm u}^i\|_{-1}^2+C_1\left(1+\frac{1}{\theta}\right)\tau+2M\tau\\
&\qquad\quad+C_1\left(1+\frac{1}{\theta}\right)T\tau+\tau+\frac{C_2^2}{\theta^2}T\tau\bigg].
\end{aligned}
\end{align}

\textbf{Step 2: Estimates for $\|e_{\bm u}^{n+1}-e_{\bm u}^n\|_{0}$.}
It remains to estimate $\|e_{\bm u}^{n+1}-e_{\bm u}^n\|_{-1}$ for $0\le n\le N-1$. Since $\|e_{\bm u}^{n+1}-e_{\bm u}^n\|_{-1}\le \|e_{\bm u}^{n+1}-e_{\bm u}^n\|_{0}$, we will bound $\|e_{\bm u}^{n+1}-e_{\bm u}^n\|_{0}$ instead.  Taking the difference of~\eqref{err_eq_u} between two consecutive time steps yields
\begin{align}\label{eq:dt}
\begin{aligned}
\frac{d_te_{\bm u}^{n+1}-d_te_{\bm u}^n}{\tau}-\nu\Delta d_te_{\bm u}^{n+1}+\nabla d_te_p^{n+1}&=d_t\bm R_{\bm u}^{n+1}+\frac{(\bm u(t_{n+1})\cdot\nabla)\bm u(t_{n+1})-(\bm u(t_n)\cdot\nabla)\bm u(t_n)}{\tau}\\
&\quad-\frac{q^{n+1}(\bm u^n\cdot\nabla)\bm u^n-q^{n}(\bm u^{n-1}\cdot\nabla)\bm u^{n-1}}{\tau}.
\end{aligned}
\end{align}
Taking the $\bm L^2$ inner product of~\eqref{eq:dt} with $d_te_{\bm u}^{n+1}$ gives us
\begin{align}\label{S1S2}
\begin{aligned}
&\frac{\|d_te_{\bm u}^{n+1}\|_0^2-\|d_te_{\bm u}^n\|_0^2}{2\tau}+\frac{\|d_te_{\bm u}^{n+1}-d_te_{\bm u}^n\|_0^2}{2\tau}+\nu\|\nabla d_te_{\bm u}^{n+1}\|_0^2\\
&\quad=(d_t\bm R_{\bm u}^{n+1},d_te_{\bm u}^{n+1})+\frac{b(\bm u(t_{n+1}),\bm u(t_{n+1}),d_te_{\bm u}^{n+1})-b(\bm u(t_n),\bm u(t_n),d_te_{\bm u}^{n+1})}{\tau}\\
&\qquad -\frac{q^{n+1}b(\bm u^n,\bm u^n,d_te_{\bm u}^{n+1})-q^nb(\bm u^{n-1},\bm u^{n-1},d_te_{\bm u}^{n+1})}{\tau}\\
&\quad=(d_t\bm R_{\bm u}^{n+1},d_te_{\bm u}^{n+1})+S_1+S_2,
\end{aligned}
\end{align}
where $S_1$ and $S_2$ are defined as
\begin{align*}
S_1&=\frac{b(\bm u(t_{n+1}),\bm u(t_{n+1}),d_te_{\bm u}^{n+1})-2b(\bm u(t_n),\bm u(t_n),d_te_{\bm u}^{n+1})+b(\bm u(t_{n-1}),\bm u(t_{n-1}),d_te_{\bm u}^{n+1})}{\tau}\\
S_2&=-\frac{q^{n+1}b(\bm u^n,\bm u^n,d_te_{\bm u}^{n+1})-b(\bm u(t_n),\bm u(t_n),d_te_{\bm u}^{n+1})}{\tau}\\
&\hspace{3cm}+\frac{q^nb(\bm u^{n-1},\bm u^{n-1},d_te_{\bm u}^{n+1})-b(\bm u(t_{n-1}),\bm u(t_{n-1}),d_te_{\bm u}^{n+1})}{\tau}.
\end{align*}
By noting that $d_te_{\bm u}^{n+1}\in\bm H_0^1(\Omega)$ and
\begin{align*}
d_t\bm R_{\bm u}^{n+1}&=\frac{1}{\tau^2}\left[\int_{t_n}^{t_{n+1}}(t-t_n)\bm u_{tt}(t)dt-\int_{t_{n-1}}^{t_n}(t-t_{n-1})\bm u_{tt}(t)dt\right]\\
&=\frac{1}{\tau^2}\int_{t_{n-1}}^{t_n}\int_r^{r+\tau}\int_s^{r+\tau}\bm u_{ttt}(\xi)\,d\xi \,ds\,dr,
\end{align*}
we obtain
\begin{align}\label{dtR}
(d_t\bm R_{\bm u}^{n+1},d_te_{\bm u}^{n+1})\le \|d_t\bm R_{\bm u}^{n+1}\|_{-1}\|d_te_{\bm u}^{n+1}\|_1\le \left(\int_{t_{n-1}}^{t_{n+1}}\|\bm u_{ttt}\|_{-1}dt\right)\|d_te_{\bm u}^{n+1}\|_1.
\end{align}
Due to the linearity of the operator $b$, $S_1$ can be written as follows
\begin{align*}
S_1=&\frac{b(\bm u(t_{n+1})-2\bm u(t_n)+\bm u(t_{n-1}),\bm u(t_{n-1}),d_te_{\bm u}^{n+1})}{\tau}\\
&\qquad+\frac{b(\bm u(t_{n}),\bm u(t_{n+1})-2\bm u(t_n)+\bm u(t_{n-1}),d_te_{\bm u}^{n+1})}{\tau}\\
&\qquad+\frac{b(\bm u(t_{n+1})-\bm u(t_n),\bm u(t_{n+1})-\bm u(t_{n-1}),d_te_{\bm u}^{n+1})}{\tau}.
\end{align*}
Using the properties of $b$ in~\eqref{b:skew}-\eqref{b:original} and the regularity assumption~\eqref{M:velo}, we find
\begin{align}\label{S1}
\begin{aligned}
S_1&\le \frac{c_2}{\tau}\|\bm u(t_{n+1})-2\bm u(t_n)+\bm u(t_{n-1})\|_0\|\bm u(t_{n-1})\|_2\|d_te_{\bm u}^{n+1}\|_1\\
&\quad+\frac{c_2}{\tau}\|\bm u(t_{n})\|_2\|\bm u(t_{n+1})-2\bm u(t_n)+\bm u(t_{n-1})\|_0\|d_te_{\bm u}^{n+1}\|_1\\
&\quad+\frac{c_2}{\tau}\|\bm u(t_{n+1})-\bm u(t_n)\|_1\|\bm u(t_{n+1})-\bm u(t_{n-1})\|_1\|d_te_{\bm u}^{n+1}\|_1\\
&\le \left(2c_2M\int_{t_{n-1}}^{t_{n+1}}\|\bm u_{tt}\|_0\,dt+2c_2\int_{t_{n-1}}^{t_{n+1}}\|\bm u_t\|_1^2\,dt\right)\|d_te_{\bm u}^{n+1}\|_1.
\end{aligned}
\end{align}
On the other hand, $S_2$ can be expressed as
\begin{align*}
S_2&=-d_te_q^{n+1}b(\bm u^n,\bm u^n,d_te_{\bm u}^{n+1})\\
&\quad-e_q^{n}\frac{b(\bm u(t_n),\bm u(t_n)-\bm u(t_{n-1}),d_te_{\bm u}^{n+1})+b(\bm u(t_n)-\bm u(t_{n-1}),\bm u(t_{n-1}),d_te_{\bm u}^{n+1})}{\tau}\\
&\quad-q^n\frac{b(e_{\bm u}^{n-1},\bm u(t_n)-\bm u(t_{n-1}),d_te_{\bm u}^{n+1})+b(\bm u(t_n)-\bm u(t_{n-1}),e_{\bm u}^{n},d_te_{\bm u}^{n+1})}{\tau}\\
&\quad -q^n\left[b(d_te_{\bm u}^{n},\bm u(t_n),d_te_{\bm u}^{n+1})+b(\bm u(t_{n-1}),d_te_{\bm u}^{n},d_te_{\bm u}^{n+1})\right]\\
&\quad -q^n\left[b(d_te_{\bm u}^{n},e_{\bm u}^{n-1},d_te_{\bm u}^{n+1})+b(e_{\bm u}^n,d_te_{\bm u}^{n},d_te_{\bm u}^{n+1})\right].
\end{align*}
By using~\eqref{b:skew}-\eqref{b:original} and~\eqref{M:velo} again, along with Lemma~\ref{range:q}, one arrives at
\begin{align}\label{S2}
\begin{aligned}
S_2&\le c_2|d_te_q^{n+1}|\|\bm u^n\|_1^2\|d_te_{\bm u}^{n+1}\|_1+\frac{2c_2M}{\tau}|e_q^n|\left(\int_{t_{n-1}}^{t_n}\|\bm u_t\|_0\,dt\right)\|d_te_{\bm u}^{n+1}\|_1\\
&\quad +\frac{c_0c_2}{\tau}(\|e_{\bm u}^{n-1}\|_0+\|e_{\bm u}^{n}\|_0)\left(\int_{t_{n-1}}^{t_n}\|\bm u_t\|_2\,dt\right)\|d_te_{\bm u}^{n+1}\|_1\\
&\quad +2c_0c_2M\|d_te_{\bm u}^{n}\|_0\|d_te_{\bm u}^{n+1}\|_1+c_0c_2(\|e_{\bm u}^{n-1}\|_1+\|e_{\bm u}^{n}\|_1)\|d_te_{\bm u}^{n}\|_1\|d_te_{\bm u}^{n+1}\|_1.
\end{aligned}
\end{align}
Collecting~\eqref{S1S2}-\eqref{S2}, we find
\begin{align}\label{before_Cauchy}
\begin{aligned}
&\frac{\|d_te_{\bm u}^{n+1}\|_0^2-\|d_te_{\bm u}^n\|_0^2}{2\tau}+\frac{\|d_te_{\bm u}^{n+1}-d_te_{\bm u}^n\|_0^2}{2\tau}+\nu\|\nabla d_te_{\bm u}^{n+1}\|_0^2\\
&\quad\le \bigg[\int_{t_{n-1}}^{t_{n+1}}\|\bm u_{ttt}\|_{-1}dt+2c_2M\int_{t_{n-1}}^{t_{n+1}}\|\bm u_{tt}\|_0\,dt+2c_2\int_{t_{n-1}}^{t_{n+1}}\|\bm u_t\|_1^2\,dt\\
&\qquad\quad+c_2|d_te_q^{n+1}|\|\bm u^n\|_1^2+\frac{2c_2M}{\tau}|e_q^n|\int_{t_{n-1}}^{t_n}\|\bm u_t\|_0\,dt+\frac{c_0c_2}{\tau}(\|e_{\bm u}^{n-1}\|_0+\|e_{\bm u}^{n}\|_0)\int_{t_{n-1}}^{t_n}\|\bm u_t\|_2\,dt\\
&\qquad\quad+2c_0c_2M\|d_te_{\bm u}^{n}\|_0+c_0c_2(\|e_{\bm u}^{n-1}\|_1+\|e_{\bm u}^{n}\|_1)\|d_te_{\bm u}^{n}\|_1\bigg]\|d_te_{\bm u}^{n+1}\|_1.
\end{aligned}
\end{align}
Using similar arguments as in~\eqref{est:u_rhs} and the facts (cf. Theorem~\ref{thm:DRLM1}) that for $0\le n\le N$,
\begin{align*}
&|e_q^n|\le \frac{C_2}{\theta}\tau,\quad\|e_{\bm u}^n\|_0^2\le C_1\left(1+\frac{1}{\theta}\right)\tau^2,\quad\|e_{\bm u}^n\|_1^2\le \frac{2(1+c_1)^2C_1}{\nu}\tau,\\
&\|\bm u^n\|_1^2\le (\|\bm u(t_n)\|_1+\|e_{\bm u}^n\|_1)^2\le 2\left(M^2+\frac{2(1+c_1)^2C_1T}{\nu}\right),
\end{align*}
we obtain from~\eqref{before_Cauchy} that
\begin{align}\label{c10}
\begin{aligned}
&\frac{\|d_te_{\bm u}^{n+1}\|_0^2-\|d_te_{\bm u}^n\|_0^2}{2\tau}+\frac{\|d_te_{\bm u}^{n+1}-d_te_{\bm u}^n\|_0^2}{2\tau}+\frac{\nu}{2}\|\nabla d_te_{\bm u}^{n+1}\|_0^2\\
&\quad\le c_{11}\bigg[\tau\int_{t_{n-1}}^{t_{n+1}}\|\bm u_{ttt}\|_{-1}^2dt+\tau\int_{t_{n-1}}^{t_{n+1}}\|\bm u_{tt}\|_0^2\,dt+\tau\int_{t_{n-1}}^{t_{n+1}}\|\bm u_t\|_1^4\,dt\\
&\qquad\qquad+|d_te_q^{n+1}|^2+\frac{\tau}{\theta^2}\int_{t_{n-1}}^{t_n}\|\bm u_t\|_0^2\,dt+\left(1+\frac{1}{\theta}\right)\tau\int_{t_{n-1}}^{t_n}\|\bm u_t\|_2^2\,dt+\|d_te_{\bm u}^{n}\|_0^2+\tau\|d_te_{\bm u}^{n}\|_1^2\bigg],
\end{aligned}
\end{align}
with
\begin{align*}
c_{11}=&\frac{(1+c_1)^2}{2\nu}\bigg[2+8c_2^2M^2+8c_2^2+4c_2^2\left(M^2+\frac{2(1+c_1)^2C_1T}{\nu}\right)^2+4C_2^2c_2^2M^2+4c_0^2c_2^2C_1\\
&\qquad\qquad+4c_0^2c_2^2M^2+\frac{8c_0^2c_2^2(1+c_1)^2C_1}{\nu}\bigg].
\end{align*}
It should be noted that the term $\tau\|d_te_{\bm u}^{n}\|_1^2$ on the right-hand side of~\eqref{c10} introduces additional complexity to the analysis, which does not occur in the 2D case~\cite{L-Shen22}.
Taking the sum of~\eqref{c10} over $i=1,2,\ldots,n$ and multiplying the resulting inequality by $2\tau$ yields
\begin{align}\label{dte1}
\begin{aligned}
&\|d_te_{\bm u}^{n+1}\|_0^2+\sum_{i=1}^n\|d_te_{\bm u}^{i+1}-d_te_{\bm u}^i\|_0^2+\tau\nu\sum_{i=1}^n\|\nabla d_te_{\bm u}^{i+1}\|_0^2\\
&\quad\le \|d_te_{\bm u}^{1}\|_0^2+2c_{11}\tau\bigg[2\tau\int_{0}^{T}\|\bm u_{ttt}\|_{-1}^2dt+2\tau\int_{0}^{T}\|\bm u_{tt}\|_0^2\,dt+2\tau\int_{0}^{T}\|\bm u_t\|_1^4\,dt\\
&\qquad\qquad\qquad\qquad\quad+\sum_{i=1}^n|d_te_q^{i+1}|^2+\frac{\tau}{\theta^2}\int_{0}^{T}\|\bm u_t\|_0^2\,dt+\left(1+\frac{1}{\theta}\right)\tau\int_{0}^{T}\|\bm u_t\|_2^2\,dt\\
&\qquad\qquad\qquad\qquad\quad+\sum_{i=1}^n\|d_te_{\bm u}^{i}\|_0^2+\tau\sum_{i=1}^n\|d_te_{\bm u}^{i}\|_1^2\bigg]\\
&\quad\le \|d_te_{\bm u}^{1}\|_0^2+2c_{11}\tau^2\|d_te_{\bm u}^{1}\|_1^2+2c_{11}\tau\bigg[6\widehat{M}\tau+\frac{C_3\tau}{\theta^2}+\frac{M\tau}{\theta^2}+\left(1+\frac{1}{\theta}\right)\widehat{M}\tau\\
&\qquad\qquad\qquad\qquad\qquad\qquad\qquad\qquad+\sum_{i=1}^n\|d_te_{\bm u}^{i}\|_0^2+\tau\sum_{i=2}^n\|d_te_{\bm u}^{i}\|_1^2\bigg],
\end{aligned}
\end{align}
where we have used Lemma~\ref{lemma:dteq} and regularity assumptions~\eqref{M:velo} and~\eqref{M:pressure} in the last estimate. Note that, by convention, the last summation on the right-hand side of~\eqref{dte1} is set to zero when $n=1$. To bound $\|d_te_{\bm u}^{1}\|_0^2+2c_{11}\tau^2\|d_te_{\bm u}^{1}\|_1^2$ in~\eqref{dte1}, we let $n=0$ in~\eqref{est:velo} and obtain
\begin{align}\label{e1}
\begin{aligned}
\frac{\|e_{\bm u}^{1}\|_0^2}{\tau}+\nu\|\nabla e_{\bm u}^{1}\|_0^2=(\bm R_{\bm u}^{1},e_{\bm u}^{1})+b(\bm u(t_{1}),\bm u(t_{1}),e_{\bm u}^{1})-q^{1}b(\bm u(t_0),\bm u(t_0),e_{\bm u}^{1}).
\end{aligned}
\end{align}
Using similar arguments as in~\eqref{est:Ru_}-\eqref{est:B4_} and the regularity of $\bm u$ in~\eqref{M:pressure}, we deduce
\begin{align}\label{e1:}
\begin{aligned}
\frac{\|e_{\bm u}^{1}\|_0^2}{\tau}+\nu\|\nabla e_{\bm u}^{1}\|_0^2&\le \left(\int_{0}^{t_{1}}\|\bm u_{tt}\|_{0}\,dt+2c_2M\int_{0}^{t_{1}}\|\bm u_t\|_1\,dt+c_2M^2|e_q^{1}|\right)\|e_{\bm u}^{1}\|_0\\
&\le \tau\left(\widehat{M}+2c_2M\widehat{M}+\frac{C_2c_2M^2}{\theta}\right)\|e_{\bm u}^{1}\|_0\le c_{12}\left(1+\frac{1}{\theta}\right)\tau\|e_{\bm u}^{1}\|_0,
\end{aligned}
\end{align}
where $c_{12}=\max\{\widehat{M}+2c_2M\widehat{M},C_2c_2M^2\}$.
By dropping the non-negative term $\nu\|\nabla e_{\bm u}^{1}\|_0^2$ on the left-hand side of~\eqref{e1:}, we find
\begin{align}\label{e1:tau2}
\begin{aligned}
\|e_{\bm u}^{1}\|_0&\le c_{12}\left(1+\frac{1}{\theta}\right)\tau^2\quad\text{ and }\quad \|d_te_{\bm u}^{1}\|_0=\frac{1}{\tau}\|e_{\bm u}^{1}\|_0\le c_{12}\left(1+\frac{1}{\theta}\right)\tau.
\end{aligned}
\end{align}
It follows from~\eqref{e1:} and~\eqref{e1:tau2} that
\begin{align}\label{tau3}
\tau^2\|d_te_{\bm u}^1\|_1^2=\|e_{\bm u}^1\|_1^2\le (1+c_1)^2\|\nabla e_{\bm u}^1\|_0^2\le \frac{(1+c_1)^2}{\nu}c_{12}^2\left(1+\frac{1}{\theta}\right)^2\tau^3.
\end{align}
From~\eqref{dte1},~\eqref{e1:tau2}, and~\eqref{tau3}, one obtains 
\begin{align}\label{c10c12}
\begin{aligned}
&\|d_te_{\bm u}^{n+1}\|_0^2+\sum_{i=1}^n\|d_te_{\bm u}^{i+1}-d_te_{\bm u}^i\|_0^2+\tau\nu\sum_{i=1}^n\|\nabla d_te_{\bm u}^{i+1}\|_0^2\\
&\quad\le c_{13}\left(1+\frac{1}{\theta}+\frac{1}{\theta^2}\right)\tau^2+2c_{11}\tau \sum_{i=1}^n\|d_te_{\bm u}^{i}\|_0^2+2c_{11}\tau^2\sum_{i=2}^n\|d_te_{\bm u}^{i}\|_1^2,
\end{aligned}
\end{align}
where $c_{13}=2c_{12}^2+\frac{4(1+c_1)^2c_{11}c_{12}^2T}{\nu}+2c_{11}(C_3+M+7\widehat{M})$.   
If $$\tau\le \frac{\nu}{2c_{11}(1+c_1)^2}=:\widehat{\tau}_0,$$  the last summation in~\eqref{c10c12} can be bounded using Poincar\'{e} inequality~\eqref{poincare:1} by
\begin{align}\label{cancel}
2c_{11}\tau^2\sum_{i=2}^n\|d_te_{\bm u}^{i}\|_1^2\le 2c_{11}(1+c_1)^2\tau^2\sum_{i=2}^n\|\nabla d_te_{\bm u}^{i}\|_0^2\le \tau\nu\sum_{i=1}^n\|\nabla d_te_{\bm u}^{i+1}\|_0^2.
\end{align}
In light of~\eqref{cancel}, we derive from~\eqref{c10c12} that
\begin{align}\label{gronwall_c12}
\|d_te_{\bm u}^{n+1}\|_0^2\le c_{13}\left(1+\frac{1}{\theta}+\frac{1}{\theta^2}\right)\tau^2+2c_{11}\tau \sum_{i=1}^n\|d_te_{\bm u}^{i}\|_0^2,\quad n\ge 1.
\end{align}
By applying Gronwall's inequality in Lemma~\ref{lemma:Gronwall} to~\eqref{gronwall_c12}, noting from~\eqref{e1:tau2} and the definition of $c_{13}$  that $\|d_te_{\bm u}^{1}\|_0^2\le c_{13}\left(1+\frac{1}{\theta}+\frac{1}{\theta^2}\right)\tau^2$, we end up with
\begin{align}\label{bound:dteu}
\frac{1}{\tau^2}\|e_{\bm u}^{n+1}-e_{\bm u}^n\|_{0}^2=\|d_te_{\bm u}^{n+1}\|_0^2\le c_{13}\left(1+\frac{1}{\theta}+\frac{1}{\theta^2}\right)e^{2c_{11}T}\tau^2,\quad n\ge 0.
\end{align}

\textbf{Step 3: Optimal convergence rate for the pressure.}
Finally, the combination of~\eqref{pr:-1},~\eqref{bound:dteu}, and the fact $\|e_{\bm u}^{i+1}-e_{\bm u}^i\|_{-1}\le \|e_{\bm u}^{i+1}-e_{\bm u}^i\|_{0}$ results in
\begin{align}\label{c14}
\begin{aligned}
\sum_{i=0}^n\|\nabla e_p^{i+1}\|_{-1}^2&\le c_{10}\bigg[c_{13}\left(1+\frac{1}{\theta}+\frac{1}{\theta^2}\right)Te^{2c_{11}T}+C_1\left(1+\frac{1}{\theta}\right)+2M\\
&\qquad\quad+C_1\left(1+\frac{1}{\theta}\right)T+1+\frac{C_2^2}{\theta^2}T\bigg]\tau\\
&\le c_{14}\left(1+\frac{1}{\theta}\right)\tau,
\end{aligned}
\end{align}
provided that $\tau\le \min\{\tau_0,\widehat{\tau}_0,\frac{\theta}{2C_2}\}$, where $c_{14}=c_{10}\left(2c_{13}Te^{2c_{11}T}+C_1+2M+C_1T+1+C_2^2T\right)$ (note that $\frac{1}{\theta^2}\le \frac{1}{\theta}$).
The proof of Theorem~\ref{thm:pressure} is thus completed  by combining~\eqref{pr:switch} and~\eqref{c14} with $C_4=c_{14}\beta^2$.
\end{proof}

\begin{remark}\label{rm:pressure}
The pressure error estimate in Theorem~\ref{thm:pressure} requires higher regularity on the exact solution than the velocity error estimate in Theorem~\ref{thm:DRLM1}. The main reason is that we estimate $\|e_{\bm u}^{n+1}-e_{\bm u}^n\|_{0}$ in Step~2 of the above proof, rather than directly bounding $\|e_{\bm u}^{n+1}-e_{\bm u}^n\|_{-1}$.
As discussed in~\cite{Shen92,Shen94}, the desired estimate $\|e_{\bm u}^{n+1}-e_{\bm u}^n\|_{-1}\le C\tau^2$ cannot be obtained using standard techniques involving Stokes operator~$A$: $D(A)=\bm V\cap\bm H^2(\Omega)\to\bm H$ due to the following incorrect inequality~\cite{Shen94}:
\begin{align}\label{Stokes_wrong}
\gamma\|\bm v\|_{-1}^2\le \left(A^{-1}\bm v,\bm v\right),\quad \forall\,\bm v\in\bm H,
\end{align}
for any constant $\gamma>0$. However, a discrete analogue of~\eqref{Stokes_wrong} holds when using suitable finite element spaces (cf.~\cite{de18,Lan25}). This observation suggests an alternative approach for the pressure error analysis in the fully discrete setting, which could require weaker regularity assumptions on the exact solution.
\end{remark}

\section{Numerical results}\label{sec:numeric}
We verify the convergence in time of the first-order DRLM scheme~\eqref{DRLM1} and demonstrate the role of the regularization parameter $\theta$ via a manufactured solution in two dimensions. Additional tests on energy stability and benchmark problems in two and three dimensions, such as lid-driven cavity flow and Kelvin-Helmholtz instability, can be found in~\cite{Doan25}. Note that various spatial discretization techniques, such as finite difference, finite element, or finite volume methods, can be used to construct fully discrete DRLM schemes. For simplicity, we adopt the finite difference approximation on a Marker-and-Cell (MAC) staggered grid with a uniform mesh size in both $x$- and $y$-directions, i.e., $h=h_x=h_y$. 

The external force $\bm f$ is computed from the following analytic solution, in which the velocity field satisfies homogeneous Dirichlet boundary conditions: \vspace{-0.1cm}
\begin{align*}
\bm u&=\begin{pmatrix}
5\sin^2(\pi x)\sin(2\pi y)e^{-t}\\
-5\sin(2\pi x)\sin^2(\pi y)e^{-t}
\end{pmatrix},\quad p=\cos(\pi x)\sin(\pi y)e^{-t}. \vspace{-0.2cm}
\end{align*}
We set the computational domain $\Omega=(0,1)^2$, the terminal time $T=1$, and the viscosity coefficient $\nu=0.1$. We vary both the time step size $\tau$ and the spatial mesh size $h$ with $\tau=2h$ so that spatial errors are negligible compared to temporal errors. The $\bm L^2$ and $\bm H^1$ errors of the velocity, $L^2$ errors of the pressure, and absolute errors of the Lagrange multiplier at the final time with different values of $\theta$ are reported in Table~\ref{table:conv}. 
We clearly observe first-order convergence for all variables. As $\theta$ increases, errors of the Lagrange multiplier (see the last column in Table~\ref{table:conv}) decrease roughly proportionally, with a tenfold increase in $\theta$ resulting in nearly a tenfold decrease in the Lagrange multiplier error. Consequently, the velocity and pressure errors are improved when $\theta$ becomes larger, especially from $\theta=0.1$ to $\theta=1$. These observations agree well with the theoretical results in Theorems~\ref{thm:DRLM1} and~\ref{thm:pressure}.

\begin{table}[!htbp]
    \centering\normalsize
    \begin{tabular}{|cllllll|}
    \hline
$\theta$&$\tau$&$h$&$\|e_{\bm  u}\|_0$&$\|e_{\bm u}\|_{1}$&$\|e_p\|_{0}$&$|e_q|$\\ \hline
\multirow{5}*{$0.1$}
&1/8  &1/16 &6.38e-02 & 9.78e-01 & 9.78e-01 & 5.51e-01\\
&1/16 &1/32 &1.66e-02 [1.94]& 2.42e-01 [2.01]& 6.47e-01 [0.60]& 4.51e-01 [0.29]\\
&1/32 &1/64 &9.64e-03 [0.78]& 1.50e-01 [0.69]& 3.50e-01 [0.89]& 2.66e-01 [0.76]\\
&1/64 &1/128&4.99e-03 [0.95]& 7.94e-02 [0.92]& 1.81e-01 [0.95]& 1.44e-01 [0.89]\\
&1/128&1/256&2.56e-03 [0.96]& 4.11e-02 [0.95]& 9.27e-02 [0.97]& 7.53e-02 [0.94]\\
\hline
\multirow{5}*{1}
&1/8  &1/16 &3.48e-02 & 3.56e-01 & 4.21e-01 & 9.99e-02\\
&1/16 &1/32 &1.29e-02 [1.43]& 1.29e-01 [1.46]& 2.08e-01 [1.02]& 5.77e-02 [0.79]\\
&1/32 &1/64 &5.51e-03 [1.23]& 5.82e-02 [1.15]& 1.02e-01 [1.03]& 3.04e-02 [0.92]\\
&1/64 &1/128&2.51e-03 [1.13]& 2.75e-02 [1.08]& 5.05e-02 [1.01]& 1.55e-02 [0.97]\\
&1/128&1/256&1.20e-03 [1.06]& 1.34e-02 [1.04]& 2.51e-02 [1.01]& 7.84e-03 [0.98]\\
\hline
\multirow{5}*{10}
&1/8  &1/16 &3.39e-02 & 3.12e-01 & 3.08e-01 & 1.06e-02\\
&1/16 &1/32 &1.26e-02 [1.43]& 1.13e-01 [1.47]& 1.50e-01 [1.04]& 5.95e-03 [0.83]\\
&1/32 &1/64 &5.28e-03 [1.25]& 4.94e-02 [1.19]& 7.34e-02 [1.03]& 3.09e-03 [0.95]\\
&1/64 &1/128&2.38e-03 [1.15]& 2.29e-02 [1.11]& 3.63e-02 [1.02]& 1.56e-03 [0.99]\\
&1/128&1/256&1.12e-03 [1.09]& 1.10e-02 [1.06]& 1.80e-02 [1.01]& 7.87e-04 [0.99]\\
\hline
\multirow{5}*{100}
&1/8  &1/16 &3.38e-02 & 3.08e-01 & 2.96e-01 & 1.07e-03\\
&1/16 &1/32 &1.26e-02 [1.42]& 1.12e-01 [1.46]& 1.44e-01 [1.04]& 5.97e-04 [0.84]\\
&1/32 &1/64 &5.26e-03 [1.26]& 4.86e-02 [1.20]& 7.04e-02 [1.03]& 3.09e-04 [0.95]\\
&1/64 &1/128&2.37e-03 [1.15]& 2.25e-02 [1.11]& 3.48e-02 [1.02]& 1.57e-04 [0.98]\\
&1/128&1/256&1.12e-03 [1.08]& 1.08e-02 [1.06]& 1.73e-02 [1.01]& 7.87e-05 [1.00]\\
\hline
\end{tabular}
    \caption{Errors and convergence rates of the velocity, pressure, and Lagrange multiplier by the first-order DRLM scheme~\eqref{DRLM1} with $\tau=2h\in\{\frac 18,\frac{1}{16},\frac{1}{32},\frac{1}{64},\frac{1}{128}\}$ and $\theta\in \{0.1,1,10,100\}$.}
    \label{table:conv} \vspace{-0.4cm}
\end{table}

\section{Conclusions}\label{sec:conclusion}
In this work, we have established rigorous temporal error estimates for the first-order DRLM scheme applied to the incompressible NS equations. By carefully addressing the nonlinear structure of the Lagrange multiplier and its coupling with the primary variables, we derived optimal convergence rates for the velocity and pressure, with all error constants exhibiting the desired dependence on the regularization parameter. The proposed analysis applies to both two- and three-dimensional settings and demonstrates theoretical robustness of the DRLM method. Numerical results confirm the theoretical findings and highlight the important role of the regularization parameter in controlling Lagrange multiplier error and enhancing overall accuracy of the DRLM method. Temporal error analysis of the second- or higher-order DRLM schemes is highly challenging and requires further investigation. Other research directions in the next step also  include fully discrete convergence analysis of DRLM schemes as well as error analysis of DRLM method developed for the coupled Cahn-Hilliard-Navier-Stokes system~\cite{Doan25b}. \vspace{0.2cm}

\noindent {\bf Acknowledgements.} 
T.-T.-P. Hoang's work is partially supported by U.S. National Science Foundation under grant number DMS-2041884.  L. Ju’s work is partially supported by U.S. National Science Foundation under grant number DMS-2409634.  R. Lan's work is partially supported by National Natural Science Foundation of China under grant number 12301531, Shandong Provincial Natural Science Fund for Excellent Young Scientists Fund Program (Overseas) under grant number 2023HWYQ-064, Shandong Provincial Youth Innovation Project under the grant number 2024KJN057 and  the OUC Scientific Research Program for Young Talented Professionals.



\end{document}